\documentclass[]{scrartcl}

\usepackage[utf8]{luainputenc}
\usepackage[USenglish]{babel}
\usepackage{csquotes}

\usepackage[a4paper,top=27mm,bottom=20mm,inner=25mm,outer=20mm]{geometry}

\usepackage[%
  backend=bibtex,bibencoding=ascii,
  style=numeric-comp,
  giveninits=true, uniquename=init, 
  natbib=true,
  url=true,
  doi=true,
  isbn=false,
  backref=false,
  maxnames=99,
  ]{biblatex}
\usepackage{amsmath}
\allowdisplaybreaks
\numberwithin{equation}{section}
\usepackage{amssymb}
\usepackage{commath}
\usepackage{mathtools}
\usepackage{bbm}
\usepackage{nicefrac}
\usepackage{subdepth}

\usepackage{siunitx}
\usepackage{algorithm}
\usepackage{algpseudocode}

\usepackage{amsthm}
\usepackage{thmtools}
\usepackage{etoolbox}
\makeatletter
\patchcmd{\thmt@setheadstyle}
 {\bgroup\thmt@space}
 {\thmt@space}
 {}{}
\patchcmd{\thmt@setheadstyle}
 {\egroup\fi}
 {\fi}
 {}{}
\makeatother
\declaretheoremstyle[
  bodyfont=\normalfont\itshape,
  headformat=\NAME\ \NUMBER\NOTE,
]{myplain}
\declaretheoremstyle[
  headformat=\NAME\ \NUMBER\NOTE,
]{mydefinition}
\newcommand{\envqed}{{\lower-0.3ex\hbox{$\triangleleft$}}}
\declaretheorem[style=myplain,numberwithin=section]{theorem}
\declaretheorem[style=myplain,numberlike=theorem]{lemma}

\usepackage[plainpages=false,pdfpagelabels,hidelinks,unicode]{hyperref}

\usepackage{color}
\usepackage{graphicx}
\usepackage[small]{caption}
\usepackage{subcaption}

\begingroup\expandafter\expandafter\expandafter\endgroup
\expandafter\ifx\csname pdfsuppresswarningpagegroup\endcsname\relax
\else
\fi

\usepackage{pgfplotstable}
\usepgfplotslibrary{external}
\pgfplotsset{compat=1.16}
\usepackage{placeins}

\usepackage{booktabs}
\usepackage{rotating}
\usepackage{multirow}

\usepackage{enumitem}

\usepackage{ifluatex}
\ifluatex
  \usepackage[no-math]{fontspec}
\else
  \usepackage[T1]{fontenc}
\fi
\usepackage{newpxtext,newpxmath}

\usepackage{xparse}

\let\epsilon\varepsilon
\let\phi\varphi
\let\rho\varrho

\usepackage{xspace}

\renewcommand{\O}{\mathcal{O}}
\renewcommand{\i}{\mathrm{i}}

\providecommand\e{}
\renewcommand{\e}{\mathrm{e}}

\providecommand\R{}
\renewcommand{\R}{\mathbb{R}}

\newcommand{\bhat}{\widehat{b}}
\newcommand{\uhat}{\widehat{u}}

\newcommand{\tol}{\ensuremath{\tau}}
\newcommand{\atol}{\ensuremath{\tau_a}}
\newcommand{\rtol}{\ensuremath{\tau_r}}

\newcommand{\EPS}{\ensuremath{\mathrm{EPS}}\xspace}
\newcommand{\EPUS}{\ensuremath{\mathrm{EPUS}}\xspace}

\NewDocumentCommand{\RK}{o m O{\the\numexpr#2-1\relax} m O{} O{} o}{%
  \IfValueTF{#1}{#1}{RK}%
  #2(#3)#4%
  \ifblank{#6}{}{\textsubscript{F}}%
  \ifblank{#5}{}{[#5]}%
  \IfValueT{#7}{#7}%
}

\renewcommand{\Re}{\operatorname{Re}}

\newcommand{\dt}{\Delta t}

\renewcommand{\div}{\operatorname{div}}

\newcommand{\I}{\operatorname{I}}

\renewcommand{\vec}[1]{\pmb{#1}}

\NewDocumentCommand{\opD}{m+g}{%
  \IfNoValueTF{#2}
    {D_{#1}}
    {D_{#1,#2}}%
}
\NewDocumentCommand{\opDsplit}{m+g}{%
  \IfNoValueTF{#2}
    {\widetilde{D}_{#1}}
    {\widetilde{D}_{#1,#2}}%
}
\NewDocumentCommand{\opM}{g}{%
  \IfNoValueTF{#1}
    {M}
    {M_{#1}}%
}
\NewDocumentCommand{\opQ}{g}{%
  \IfNoValueTF{#1}
    {Q}
    {Q_{#1}}%
}
\NewDocumentCommand{\opI}{g}{%
  \IfNoValueTF{#1}
    {I}
    {I_{#1}}%
}
\NewDocumentCommand{\opV}{g}{%
  \IfNoValueTF{#1}
    {V}
    {V_{#1}}%
}
\NewDocumentCommand{\opB}{g}{%
  \IfNoValueTF{#1}
    {B}
    {B_{#1}}%
}
\NewDocumentCommand{\opR}{g}{%
  \IfNoValueTF{#1}
    {R}
    {R_{#1}}%
}
\NewDocumentCommand{\opN}{m+g}{%
  \IfNoValueTF{#2}
    {N_{#1}}
    {N_{#1,#2}}%
}

\NewDocumentCommand{\fnum}{g}{%
  \IfNoValueTF{#1}
    {f^{\mathrm{num}}}
    {f^{\mathrm{num,#1}}}%
}
\NewDocumentCommand{\vecfnum}{g}{%
  \IfNoValueTF{#1}
    {\vec{f}^{\mathrm{num}}}
    {\vec{f}^{\mathrm{num,#1}}}%
}
\NewDocumentCommand{\vecfcorr}{g}{%
  \IfNoValueTF{#1}
    {\vec{f}^{\mathrm{corr}}}
    {\vec{f}^{\mathrm{corr,#1}}}%
}
\NewDocumentCommand{\fvol}{g}{%
  \IfNoValueTF{#1}
    {f^{\smash{\mathrm{vol}}}}
    {f^{\smash{\mathrm{vol,#1}}}}%
}

\newcommand{\orcid}[1]{ORCID:~\href{https://orcid.org/#1}{#1}}
\usepackage{authblk}

\newenvironment{keywords}{\par\textbf{Key words.}}{\par}
\newenvironment{AMS}{\par\textbf{AMS subject classification.}}{\par}

\title{Stability of step size control based on a posteriori error estimates}

\author[1]{Hendrik~Ranocha\thanks{\orcid{0000-0002-3456-2277}}}
\affil[1]{Institute of Mathematics, Johannes Gutenberg University Mainz, Germany}

\author[2]{Jan Giesselmann\thanks{\orcid{0009-0008-0217-7244}}}
\affil[2]{Numerical Analysis and Scientific Computing, Technical University of Darmstadt, Germany}

\date{December 11, 2023} 

\makeatletter
\hypersetup{pdfauthor={Hendrik Ranocha, Jan Giesselmann}} 
\hypersetup{pdftitle={Stability of step size control based on a posteriori error estimates}} 
\makeatother

\begin{document}

\maketitle

\begin{abstract}
\noindent
  A posteriori error estimates based on residuals can be used for reliable error
control of numerical methods. Here, we consider them in the context of ordinary
differential equations and Runge-Kutta methods. In particular, we take the
approach of Dedner \& Giesselmann (2016) and investigate it when used to select
the time step size. We focus on step size control stability when combined with
explicit Runge-Kutta methods and demonstrate that a standard I controller is
unstable while more advanced PI and PID controllers can be designed to be
stable. We compare the stability properties of residual-based
estimators and classical error estimators based on an embedded Runge-Kutta method
both analytically and in numerical experiments.

\end{abstract}



\section{Introduction}
\label{sec:introduction}

A posteriori error estimators come in different varieties; reliability, efficiency and asymptotic exactness being frequent quality criteria. The diversity in error estimators reflects the fact that they can be used for two different purposes: \emph{error control} and \emph{step size selection}. Both purposes are connected but place an emphasis on different properties.
When it comes to step size selection in adaptive numerical methods, low computational costs are paramount; this goal is achieved for explicit Runge-Kutta schemes by embedded schemes used in extrapolation mode \cite{dormand1980family}. However, this methodology provides no error control, i.e., the user cannot certify whether a given numerical simulation is compatible with some error tolerance or not.
Reliable (and efficient) error estimators have as their primary objective to control the error but can also be used to compute provably quasi-optimal meshes in elliptic and parabolic problems \cite{becker2023,kreuzer2012}. Such methods are, commonly, not used in step size control of (explicit) schemes for ordinary differential equations
(ODEs) due to the larger computational costs. In particular, they usually require to measure how much the numerical solution fails to satisfy the ODE, by the so called residual, and to relate the residual to the error by a suitable stability theory which might be based on energy or duality arguments \cite{lakkis2015}.

Nevertheless, if one decides to compute residuals  in order to ensure error control, it makes sense to also use this information for choosing step sizes and it is desirable that this leads to stable step size control.
Typically, the most simple error-based step size selection uses an I controller
that multiplies the current time step size by a factor derived from an
error estimate using asymptotic arguments. By construction, it usually works
well in this asymptotic regime of small time step sizes. However, explicit
Runge-Kutta methods also need to operate well when the time step size is
limited by stability instead of accuracy. In this situation, step size control
stability is important. The study of these properties has been initiated by
Hall \cite{hall1985equilibrium,hall1986equilibrium} with further refinements
and applications together with Higham
\cite{hall1988analysis,higham1990embedded}.

One option to obtain step size control stability when using embedded
Runge-Kutta methods such as the classical schemes of
Bogacki and Shampine \cite{bogacki1989a32,bogacki1996efficient}
or Dormand and Prince \cite{dormand1980family} is to design the methods
specifically to allow step size control stability with the classical
I controller \cite{higham1990embedded}. However, most schemes of this class
used nowadays make use of more advanced controllers such as PI and PID
controllers developed for example in
\cite{gustafsson1988pi,gustafsson1991control,soderlind2006time,soderlind2006adaptive}.
As demonstrated in \cite{ranocha2023error,ranocha2021optimized}, these
controllers can be used together with embedded Runge-Kutta method for
efficient and robust time step size control in the context of compressible
computational fluid dynamics where the stability limited regime is crucial
due to the Courant-Friedrichs-Lewy step size restriction \cite{courant1967partial}.

In the following Section~\ref{sec:basics}, we introduce the notation
and basic ideas of the methods. Next, we investigate the step size control
stability of methods derived from the residual-based a posteriori error
estimators of \cite{dedner2016posteriori} analytically in
Section~\ref{sec:control-stability} and numerically in
Section~\ref{sec:numerical_experiments}. Finally, we summarize and discuss
our results in Section~\ref{sec:summary}. All source code required to
reproduce the numerical experiments is available online in our
reproducibility repository \cite{ranocha2023stabilityRepro}.

\section{Basic ideas of step size control and a posteriori error estimates}
\label{sec:basics}

Consider a system of  ODEs
\begin{equation}
\label{eq:ode}
  u'(t) = f\bigl( t, u(t) \bigr),
  \quad
  u(0) = u^0 \in \R^m.
\end{equation}
One step of an explicit Runge-Kutta method can be written as
\cite{hairer2008solving,butcher2016numerical}
\begin{eqnarray}
\label{eq:RK-stages}
  y^i
  &=&
  u^n + \dt_n \sum_{j=1}^{i-1} a_{ij} \, f(t^n + c_j \dt_n, y^j),
  \qquad i \in \{1, \dots, s\},
  \\
  \label{eq:RK-step}
  u^{n+1}
  &=&
  u^n + \dt_n \sum_{i=1}^{s} b_{i} \, f(t^n + c_i \dt_n, y^i),
\end{eqnarray}
where $y^i$ are the stage values, $u^{n}$ is the numerical solution
approximating $u$ at time $t^n$, and $t^{n+1} = t^n + \dt_n$.
As usual, we assume that the row-sum condition
$\forall i\colon c_i = \sum_{j} a_{ij}$ is satisfied so that it
suffices to consider autonomous problems.

Given an approach to estimate the error made in one step and a tolerance,
a common approach is to compute a weighted error estimator $w_{n+1}$
of the form ``error estimate divided by tolerance''
\cite[Section~II.4]{hairer2008solving}.
Let $|e_{n+1}|$ be an error estimate. If only an absolute tolerance $\tau$
is used (and no relative tolerance), the weighted error estimator would
simply be $w_{n+1} = |e_{n+1}| / \tau$.
Thus, $w_{n+1} \le 1$ means that the desired tolerance is achieved.
Setting $ \epsilon_{n+1} = \frac{1}{w_{n+1}}$, classical methods for
choosing time step sizes are based on I, PI, and PID controllers, e.g.,
\cite{soderlind2006time,soderlind2006adaptive,kennedy2000low}
\begin{equation}
\label{eq:PID}
  \dt_{n+1} = \kappa\Bigl(%
                  \epsilon_{n+1}^{\beta_1 / k}
                  \epsilon_{n }^{\beta_2 / k}
                  \epsilon_{n-1}^{\beta_3 / k} \Bigr) \dt_{n},
\end{equation}
where $k$ is chosen such that $w_{n+1}$ is expected to be of order $\Delta t_n^k$.
The function $\kappa$ is a step size limiter,
which we choose as $\kappa(a) = 1 + \arctan(a - 1)$ \cite{soderlind2006adaptive}.
The real numbers $\beta_i$ are the controller parameters.
The general form \eqref{eq:PID} of a PID controller reduces to a PI controller
for $\beta_3 = 0$ and to a classical I controller for $\beta_2 = \beta_3 = 0$.
There are different ways in which the estimator $w_{n+1}$ can be obtained.

A classical approach to step size control is to obtain an error estimate
via an embedded method that consists of \eqref{eq:RK-stages} and
\begin{equation}
  \uhat^{n+1} = u^n + \dt_n \sum_{i=1}^{s} \bhat_{i} \, f(t^n + c_i \dt_n, y^i) + \dt_n \bhat_{s+1} f(t^{n+1}, u^{n+1}).
\end{equation}
Typically, these methods are used in local extrapolation mode, i.e., the main
method is of order~$p$ and the embedded method is of order $\widehat{p} = p - 1$.
Then, $k$ is chosen as $k= \min(p, \widehat{p}) + 1$, i.e., $k = p$ due to
the local extrapolation mode.
If $\bhat_{s+1} \ne 0$, the right-hand side of the new approximation
$u^{n+1}$ is used. This first-same-as-last (FSAL) technique was introduced to
improve the performance of the error estimator $u^{n+1} - \uhat^{n+1}$
\cite{dormand1980family}. Overall, one obtains
\begin{equation}
\label{eq:weighted-error-estimate-embedded}
  w_{n+1} = \left( \frac{1}{m} \sum_{i=1}^{m} \left( \frac{u_i^{n+1} - \uhat_i^{n+1}}{\atol + \rtol \max\{ \vert u_i^{n+1}\vert, \vert u_i^{n} \vert \}} \right)^2 \right)^{1/2},
\end{equation}
where  $\atol, \rtol > 0$ are absolute and relative
tolerances and  $m$ is the total number of degrees of freedom in $u$,
cf.\ \cite[Section~II.4]{hairer2008solving}.

While step size control based on embedded Runge-Kutta schemes is highly efficient (for explicit schemes) and works very well in practice it has the (theoretical) drawback that it does not provide any rigorous upper bounds for the error.

In contrast, the error estimators described in the next section lead to provable upper bounds for the error but are more expensive to compute.

\subsection{Energy based a posteriori error estimates}
\label{sec:a-posteriori-error-estimates}

The basic idea of this type of error estimators is to compute a sufficiently regular reconstruction $\widehat u$ from the numerical solution, to compute the residual
\begin{equation}\label{eq:res}
  R := \frac{\dif}{\dif t} \widehat{u} - f(\widehat{u}),
\end{equation}
and to use a suitable stability theory of the ODE to bound the difference between $u$ and $\widehat{u}$ in terms of $R$.
Indeed, if $f$ satisfies a one-sided Lipschitz condition, i.e. there exists $L\in \mathbb{R}$ such that for all $u,v \in \mathbb{R}^m$
\begin{equation}\label{eq:osL}
 \langle f(u) - f(v) , u -v \rangle \leq L \| u- v\|^2,
\end{equation}
then
Gronwall's lemma implies, in the  (scaled) Euclidean norm on $\mathbb{R}^m$,
\begin{lemma}
\label{lem:estimate}
 Let $u$ be a solution of \eqref{eq:ode} and let $\widehat u$ solve \eqref{eq:res}. Then, for all $0 \leq t \leq T$
\begin{equation}\label{eq:gron}
 \| u(t) - \widehat u(t)\| \leq \left(  \| u(0) - \widehat u(0)\|
 + \| R \exp(-L \cdot) \|_{L^1(0,T)} \right) \e ^{Lt}
\end{equation}
where $L$ is such that \eqref{eq:osL} holds.
If $\widehat u$ is the reconstruction of a numerical solution satisfying $\widehat{u}(t^n) = u^n$, then as an immediate consequence for all $0 \leq t^n \leq T$
\[   \| u(t^n) - u^n\| \leq \left(  \| u(0) - \widehat u(0)\|
 + \| R\exp(-L \cdot)\|_{L^1(0,T)} \right) \e ^{LT}.\]
\end{lemma}
This provides an a posteriori error bound once it is clear how $\widehat u$ is computed from the numerical solution since $R$ can be computed by evaluating \eqref{eq:res}.

\begin{proof}[Proof of Lemma~\ref{lem:estimate}]

We multiply \eqref{eq:ode} minus \eqref{eq:res} by $u-\widehat u$  and obtain
\[ \frac12 \frac{\dif}{\dif t} \|u-\widehat u\|^2 = \langle f(u) - f(\widehat u),u-\widehat u \rangle - \langle R , u- \widehat{u} \rangle \leq L \|u-\widehat u\|^2 + \| R \|\|u-\widehat u\|. \]
Thus, setting $y_1(t)  := \exp(-Lt)  \|u(t)-\widehat u(t)\|$ we get
\[
 \frac{\dif}{\dif t} y_1^2(t) \leq 2 \|R(t)\|\exp(-2Lt) \|u(t)-\widehat u(t)\| = 2 \|R(t)\| \exp(-Lt) y_1(t)
\]
such that integrating in time leads to
\[
 \frac 12 y_1^2(t) \leq \frac 12 y_1^2(0) + \int_0^t \|R(s)\| \exp(-Ls) y_1(s) \dif s.
\]
Invoking \cite[Theorem 5]{Dragomir_2003} leads to
\[
 y_1(t) \leq y_1(0) + \int_0^t \|R(s)\| \exp(-Ls) \dif s.
\]
Inserting the definition of $y_1$ and multiplying by $\exp(Lt)$ gives the desired result.
\end{proof}

If $\| R \exp(-L \cdot) \|_{L^1(0,T)}$ is computed for error control,  it seems reasonable to choose step sizes based on $\| R\|_{L^1(t^n,t^{n+1})}$.
However, certain care is  needed in this approach: First of all, it needs to be ensured by a suitable reconstruction strategy that $\| R\|_{L^1(t^n,t^{n+1})}$ is indeed of order $\Delta t_n^{p+1}$ for a $p$-th order scheme and sufficiently regular solutions. Defining such a reconstruction is not trivial since the stage values of RK schemes do not contain high order information directly and computing the residual involves taking a derivative which might lead to the loss of one order of convergence. Indeed, it turns out that, in general, for any numerical scheme a dedicated reconstruction method needs to be derived in order to ensure that $\| R\|_{L^1(t^n,t^{n+1})}$ is of order $\Delta t_n^{p+1}$.

It should also be kept in mind that there are many interesting scenarios where the step from \eqref{eq:res} to \eqref{eq:gron} is far less straightforward. One scenario is that no $L$ exists such that \eqref{eq:osL} holds. Another scenario that frequently arises when $f$ stems from the spatial discretization of a PDE is that $L$ is positive and depends on an inverse power of the spatial mesh width. In such cases a more sophisticated stability analysis is needed in order to connect errors and residuals in an optimal way.

This line of thought can be traced back, at least, to \cite{makridakisnochetto2006} and was expunged in detail in \cite{makridakis2007}. These works address discretizations of parabolic PDEs and implicit time discretizations. It turns out that in case explicit RK schemes are applied to ODEs or semi-discretizations of first order hyperbolic PDEs a more generic method based on Hermite interpolation can be used \cite{dedner2016posteriori}.
Indeed, inspection of the proof of \cite[Theorem 2.1]{dedner2016posteriori}, which only covers equidistant time steps, shows that the reconstructions of Hermite-type considered here satisfy
\begin{equation}\label{eq:opt}\| R \|_{L^\infty(t^n, t^{n+1})} \leq C \Delta t_n^{p} \end{equation}
provided the Runge-Kutta scheme and $f$ are such that the consistency error is bounded by $\tilde C \Delta t_n^{p+1}$.
Detailed formulas of these reconstructions for generic cases can be found in \cite{dedner2016posteriori}.
Note that  \cite{dedner2016posteriori} also considers reconstructions using values from previous time steps and for those \eqref{eq:opt} will not hold, since the left-hand side also depends on sizes of previous time steps.
All reconstructions considered in this paper only use values from the current time step and we will provide explicit formulas for the reconstructions we investigate in this paper, e.g., \eqref{eq:cubic-hermite-central}.

A key observation of \cite{dedner2016posteriori} is that $f(t^n, u^n)$ is not only readily available in building the reconstruction since it is anyway computed during the time step but also known to be accurate enough in this scenario.
Step size control for explicit RK schemes and the estimators from  \cite{dedner2016posteriori} is what we are going to investigate in this paper.

In the context of these methods we
evaluate the weighted error estimator either in the $L^1$ norm as
\begin{equation}
\label{eq:weighted-estimate-residual-L1}
  w_{n+1} = \frac{\| R \|_{L^1(t^n, t^{n+1})}}{\atol + \rtol \max\{ \|u^n\|, \|u^{n+1}\| \}}
\end{equation}
or in the $L^2$ norm as
\begin{equation}
\label{eq:weighted-estimate-residual-L2}
  w_{n+1} = \frac{\sqrt{\dt_n} \| R \|_{L^2(t^n, t^{n+1})}}{\atol + \rtol \max\{ \|u^n\|, \|u^{n+1}\| \}}.
\end{equation}
Note the multiplication by $\sqrt{\dt_n}$ in
\eqref{eq:weighted-estimate-residual-L2}
ensures the correct scaling in terms of the time step size $\dt_n$.
Since the estimator is of the same order as the main method, we use $k=p+1$.

\subsection{Error per step versus error per unit step}
\label{sec:EPS-vs-EPUS}

There are several interesting properties of step size controllers.
While the main focus of this paper is \emph{step size control stability},
discussed in Section ~\ref{sec:control-stability}, the current section discusses
two other important properties: \emph{tolerance convergence} and the stronger
property \emph{tolerance proportionality}.
\emph{Tolerance convergence} means
$\tau \to 0 \implies \| E \| \to 0$,
i.e., to obtain a global error $E$ converging to zero when the tolerance $\tau$
goes to zero.
\emph{Tolerance proportionality} means
$\exists c, C > 0\colon c \tau \le \|E\| \le C \tau$ for all sufficiently
small $\tau > 0$,
i.e., to obtain a global error that is roughly proportional to the
given tolerance.
In this section, we discuss these properties for \emph{error per step} (\EPS)
and \emph{error per unit step} (\EPUS) control based on
\eqref{eq:weighted-estimate-residual-L1} or \eqref{eq:weighted-estimate-residual-L2}.

The formulas \eqref{eq:weighted-estimate-residual-L1} and
\eqref{eq:weighted-estimate-residual-L2} estimate the \emph{error per step} and
adaptation as in \eqref{eq:PID} aims at keeping it below a given tolerance.
An alternative is to control the \emph{error per unit step}.
Then, the tolerance should control the error per step estimate $|e_{n+1}|$
divided by the step size $\dt_n$, i.e., $\tau \ge |e_{n+1}| / \dt_n$
instead of $\tau \ge |e_{n+1}|$ for \EPS control.
In this case, the controllers still have the same form as above but use
$\epsilon_{n+1} = \dt_n / w_{n+1}$ instead of
$\epsilon_{n+1} = 1 / w_{n+1}$.

Higham \cite{higham1991global} analyzed tolerance proportionality for
different control strategies. He proved that tolerance proportionality
is obtained for a $p$th-order method if a local error estimate
scaling as $\O(\dt_n^p)$ is chosen as control objective
\cite[Corollary~2.1]{higham1991global}. In particular, this is the case
if the local error per step is controlled for an explicit Runge-Kutta pair
in local extrapolation mode, i.e., where a main method of order $p$ is
coupled with an embedded method of order $\widehat{p} = p - 1$.
This is also the case when the local error per unit step is controlled based on
an error estimate $|e_{n+1}| = \O(\dt_n^{p+1})$
--- which is the scaling of the
residual-based error estimators described in
Section~\ref{sec:a-posteriori-error-estimates}.

Butcher \cite[Sections~371--373]{butcher2016numerical} argues
that an \EPS control is close to producing ``optimal'' step size sequences
(in a sense discussed there). Thus,  an \EPS control can be considered to be better than
an \EPUS control, even if the  \EPUS control provides  tolerance proportionality
and the \EPS does not.

Moreover, tolerance convergence can still be expected from an \EPS
control if the corresponding \EPUS control leads to tolerance
proportionality. To see this, let $|e_{n+1}|$ be the local error estimate
and $\tau$ the tolerance.
The \EPUS control goal is to achieve $|e_{n+1}| / \dt \le \tau_\EPUS$.
The \EPS control has the goal $|e_{n+1}| \le \tau_\EPS$.
Thus, \EPS and \EPUS control have the same goal if
$\tau_\EPS = \tau_\EPUS \dt$.
For a $p$th-order method with tolerance proportionality, the step size
will scale as $\dt \propto \tau_\EPUS^{1 / p}$.
Thus, tolerance proportionality
can be achieved for \EPS control with rescaled tolerance
$\tau_\EPS = \tau_\EPUS \dt \propto \tau_\EPUS \tau_\EPUS^{1 / p}
= \tau_\EPUS^{(p + 1) / p}$. Re-arranging, we can expect to obtain
a global error scaling as $\tau_\EPS^{p / (p + 1)}$, still leading to
 tolerance convergence.
Butcher \cite[Section~373]{butcher2016numerical} argues that this is
 fine in practice --- due to the link of \EPS control
to ``optimal'' step size sequences. Such a behavior does indeed occur
in practice for well-known numerical integrators such as \textsc{Dassl}
\cite{soderlind2006adaptive}.
Since \EPS control is common for the classical approach of using an
embedded method with local extrapolation, we concentrate on this mode
in the following section.

\section{Step size control stability}
\label{sec:control-stability}

In this section, we analyze step size control stability for residual
error estimators and compare it to the situation for embedded methods.
We follow the presentation of \cite[Section~IV.2]{hairer2010solving}
to introduce the concept of step size control stability.
Consider the scalar test problem
\begin{equation}
\label{eq:test-problem}
  u'(t) = \lambda u(t),
  \quad
  u(0) = u^0,
\end{equation}
for $\lambda \in \mathbb{C}$ and a simplified I controller\footnote{An I controller is a PID controller with
$\beta_2 = \beta_3 = 0$. Here, we got rid of the step size limiter and set the
first parameter $\beta_1 = 1$. This is a classical deadbeat controller.} of the
form
\begin{equation}\label{eq:eps}
  \dt_{n+1} = \epsilon_{n+1}^{1 / k} \dt_n,
  \qquad
  \epsilon_{n+1} = \frac{\tol}{|e_{n+1}|},
\end{equation}
where $\tau$ is a fixed tolerance and $|e_{n+1}|$ is the error estimate, e.g.,
$|e_{n+1}| = | u^{n+1} - \uhat^{n+1} |$ when an embedded method is used.
This yields the dynamical system
\begin{equation}
\begin{aligned}
  u^{n+1} &= R(\dt_n \lambda) u^{n},
  \\
  \dt_{n+1} &= \dt_n \left( \frac{\tau}{|e_{n+1}|} \right)^{1 / k},
\end{aligned}
\end{equation}
where $R$ is the stability function of the (main) Runge-Kutta method.
Here, we have assumed that the error estimate $|e_{n+1}|$ depends
only on $u^{n}$ and $\dt_n$ so that no additional equation
for $|e_{n+1}|$ is required.
This assumption holds for all error estimates considered in this article.
The analysis can be simplified by introducing logarithms
\begin{equation}
  \eta_n = \log| u^{n} |, \quad \chi_n = \log \dt_n,
\end{equation}
resulting in
\begin{equation}
\label{eq:dyn-sys-I}
\begin{aligned}
  \eta_{n+1} &= \log| R(\e^{\chi_n} \lambda) | + \eta_n,
  \\
  \chi_{n+1} &= \chi_n + \frac{1}{k} \left( \log(\tol) - \log|e_{n+1}| \right).
\end{aligned}
\end{equation}
To study the step size control stability, we investigate the stability properties
of fixed points defined by
\begin{equation}
  | R(\e^{\chi_n} \lambda) | = 1,
  \quad
  \log|e_{n+1}| = \log(\tol).
\end{equation}
The first equation states that the step size $\dt_n = \e^{\chi_n}$ is chosen
such that $z = \dt_n \lambda$ is on the boundary of the stability region of the
Runge-Kutta method. A stable behavior requires that the spectral radius of the
Jacobian
\begin{equation}
\label{eq:jacobian-I}
  J =
  \begin{pmatrix}
    1 & \mu \\
    -\frac{1}{k} \partial_{\eta_{n}} \log|e_{n+1}| &
      1 - \frac{1}{k} \partial_{\chi_{n}} \log|e_{n+1}|
  \end{pmatrix},
  \qquad
  \mu = \Re\biggl( \frac{R'(z)}{R(z)} z \biggr),
\end{equation}
does not exceed unity.

For a simplified PID controller of the form
\begin{equation}
  \dt_{n+1} = \epsilon_{n+1}^{\beta_1 / k}
              \epsilon_{n  }^{\beta_2 / k}
              \epsilon_{n-1}^{\beta_3 / k} \dt_n,
  \qquad
  \epsilon_{n+1} = \frac{\tol}{|e_{n+1}|},
\end{equation}
the dynamical system becomes
\begin{equation}
\label{eq:dyn-sys-PID}
\begin{aligned}
  \eta_{n+1} &= \eta_n + \log| R(\e^{\chi_n} \lambda) |,
  \\
  \chi_{n+1} &= \chi_n
              + \frac{\beta_1}{k} \left( \log(\tol) - \log|e_{n+1}| \right)
              + \frac{\beta_2}{k} \left( \log(\tol) - \log|e_{n  }| \right)
  \\
  &\quad\quad\;\,
              + \frac{\beta_3}{k} \left( \log(\tol) - \log|e_{n-1}| \right).
\end{aligned}
\end{equation}
This can be considered as a dynamical system mapping from the indices
$(n, n-1, n-2)$ to the indices $(n+1, n, n-1)$. The corresponding Jacobian is
\begin{equation}
\label{eq:jacobian-PID}
\begin{gathered}
  J =
  \begin{pmatrix}
    1 & \mu & 0 & 0 & 0 & 0
    \\
    -\frac{\beta_1}{k} \frac{\partial l_{n+1}}{\partial \eta_{n}} &
      1 - \frac{\beta_1}{k} \frac{\partial l_{n+1}}{\partial \chi_{n}} &
      -\frac{\beta_2}{k} \frac{\partial l_{n}}{\partial \eta_{n-1}} &
      - \frac{\beta_2}{k} \frac{\partial l_{n}}{\partial \chi_{n-1}} &
      -\frac{\beta_3}{k} \frac{\partial l_{n-1}}{\partial \eta_{n-2}} &
      - \frac{\beta_3}{k} \frac{\partial l_{n-1}}{\partial \chi_{n-2}}
    \\
    1 & 0 & 0 & 0 & 0 & 0
    \\
    0 & 1 & 0 & 0 & 0 & 0
    \\
    0 & 0 & 1 & 0 & 0 & 0
    \\
    0 & 0 & 0 & 1 & 0 & 0
  \end{pmatrix},
  \\
  l_{j} = \log|e_{j}|,
  \quad
  \mu = \Re\biggl( \frac{R'(z)}{R(z)} z \biggr),
\end{gathered}
\end{equation}
cf. \cite{kennedy2000low,ranocha2021optimized}.
When an embedded method
with stability function $\widehat{R}$ is used to estimate the error,
we consider the polynomial $E(z) = R(z) - \widehat{R}(z)$. Then,
the Jacobian \eqref{eq:jacobian-PID} becomes
\cite{kennedy2000low,ranocha2021optimized}
\begin{equation}
\label{eq:jacobian-PID-embedded}
\begin{gathered}
  J =
  \begin{pmatrix}
    1 & \mu & 0 & 0 & 0 & 0
    \\
    -\frac{\beta_1}{k} &
      1 - \frac{\beta_1}{k} \nu &
      -\frac{\beta_2}{k} &
      - \frac{\beta_2}{k} \nu &
      -\frac{\beta_3}{k} &
      - \frac{\beta_3}{k} \nu
    \\
    1 & 0 & 0 & 0 & 0 & 0
    \\
    0 & 1 & 0 & 0 & 0 & 0
    \\
    0 & 0 & 1 & 0 & 0 & 0
    \\
    0 & 0 & 0 & 1 & 0 & 0
  \end{pmatrix},
  \\
  \mu = \Re\biggl( \frac{R'(z)}{R(z)} z \biggr),
  \quad
  \nu = \Re\biggl( \frac{E'(z)}{E(z)} z \biggr).
\end{gathered}
\end{equation}
In the following, we will study step size control stability for several
explicit Runge-Kutta methods from first to third order of accuracy. We begin
with the explicit Euler method to illustrate the steps. The other calculations
use Mathematica \cite{mathematica12}; the corresponding notebooks are
available in our reproducibility repository \cite{ranocha2023stabilityRepro}.

\subsection{Explicit Euler method}

The most simple explicit Runge-Kutta method is the explicit Euler method
with stability function
\begin{equation}
  R(z) = 1 + z.
\end{equation}
We use the linear reconstruction polynomial
\begin{equation}
  \widehat{u}(t) = u^{n} + \frac{t - t^{n}}{t^{n+1} - t^{n}} (u^{n+1} - u^{n})
\end{equation}
for the time interval $[t^{n}, t^{n+1}]$.
Then, the weighted $L^1$ error estimate \eqref{eq:weighted-estimate-residual-L1}
is given by
\begin{equation}
\begin{aligned}
  |e_{n+1}|
  =
  \| R \|_{L^1(t^n, t^{n+1})}
  &=
  \int_{t^{n}}^{t^{n+1}} \left|
    \frac{\dif}{\dif t} \widehat{u}(t) - \lambda \widehat{u}(t)
  \right| \dif t
  \\
  &=
  \int_{t^{n}}^{t^{n+1}} (t - t^{n}) \dif t \;
  |\lambda|^2 |u^{n}|
  =
  \frac{1}{2} \dt_n^2 |\lambda|^2 |u^{n}|.
\end{aligned}
\end{equation}
The $L^2$ version \eqref{eq:weighted-estimate-residual-L1} uses
\begin{equation}
  |e_{n+1}|
  =
  \sqrt{\dt_n} \| R \|_{L^2(t^n, t^{n+1})}
  =
  \frac{1}{\sqrt{3}} \dt_n^2 |\lambda|^2 |u^{n}|.
\end{equation}
Thus, the Jacobian \eqref{eq:jacobian-I} of the I controller becomes in both
cases
\begin{equation}
  J =
  \begin{pmatrix}
    1 & \mu \\
    -\frac{1}{k} &
      1 - \frac{2}{k}
  \end{pmatrix},
  \qquad
  \mu = \Re\biggl( \frac{z}{1 + z} \biggr).
\end{equation}

\subsection{Second-order, two-stage methods}
\label{sec:RK22}

All explicit second-order, two-stage Runge-Kutta methods have the stability
function
\begin{equation}
  R(z) = 1 + z + \frac{z^2}{2}.
\end{equation}
We use the left-biased quadratic Hermite interpolation polynomial
\begin{equation}
  \widehat{u}(t)
  =
  \biggl( 1 - \frac{t^2}{\dt_n^2} \biggr) u^{n} +
  \biggl( t - \frac{t^2}{\dt_n} \biggr) f(t^{n}, u^{n}) +
  \frac{t^2}{\dt_n^2} u^{n+1}
\end{equation}
normalized to $t \in [0, \dt_n]$.
Then, the weighted $L^1$ error estimate \eqref{eq:weighted-estimate-residual-L1}
is given by
$w_{n+1} = |e_{n+1}| / (\atol + \rtol \max\{ \|u^n\|, \|u^{n+1}\| \})$
with
\begin{equation}
\begin{aligned}
  |e_{n+1}|
  =
  \| R \|_{L^1(t^n, t^{n+1})}
  &=
  \int_{t^{n}}^{t^{n+1}} \left|
    \frac{\dif}{\dif t} \widehat{u}(t) - \lambda \widehat{u}(t)
  \right| \dif t
  \\
  &=
  \frac{1}{2} \int_{0}^{\dt_n} t^2 \dif t \;
  |\lambda|^3 |u^{n}|
  =
  \frac{1}{6} \dt_n^3 |\lambda|^3 |u^{n}|.
\end{aligned}
\end{equation}
The corresponding $L^2$ version uses
\begin{equation}
  |e_{n+1}|
  =
  \sqrt{\dt_n} \| R \|_{L^2(t^n, t^{n+1})}
  =
  \frac{1}{2 \sqrt{5}} \dt_n^3 |\lambda|^3 |u^{n}|.
\end{equation}
The expression of the Jacobian \eqref{eq:jacobian-I} of the I controller with
residual error estimator becomes in both cases
\begin{equation}
  \label{eq:jacobian-RK22-I-residual}
  J =
  \begin{pmatrix}
    1 & \mu \\
  -\frac{1}{k} & 1 - \frac{3}{k} \\
  \end{pmatrix},
  \quad
  \mu = \Re\biggl( \frac{z + z^2}{1 + z + z^2/2} \biggr).
\end{equation}
The corresponding Jacobian \eqref{eq:jacobian-PID-embedded} based on an
embedded explicit Euler method
$\widehat{u}^{n+1} = u^n + \dt_n f(t^n, u^n)$ is
\begin{equation}
\label{eq:jacobian-RK22-I-embedded}
  J =
  \begin{pmatrix}
    1 & \mu \\
  -\frac{1}{k} & 1 - \frac{2}{k} \\
  \end{pmatrix},
  \qquad
  \mu = \Re\biggl( \frac{z + z^2}{1 + z + z^2/2} \biggr),
  \;
  k = 2.
\end{equation}

\begin{figure}[htb]
\centering
  \begin{subfigure}{0.49\textwidth}
    \includegraphics[width=\textwidth]{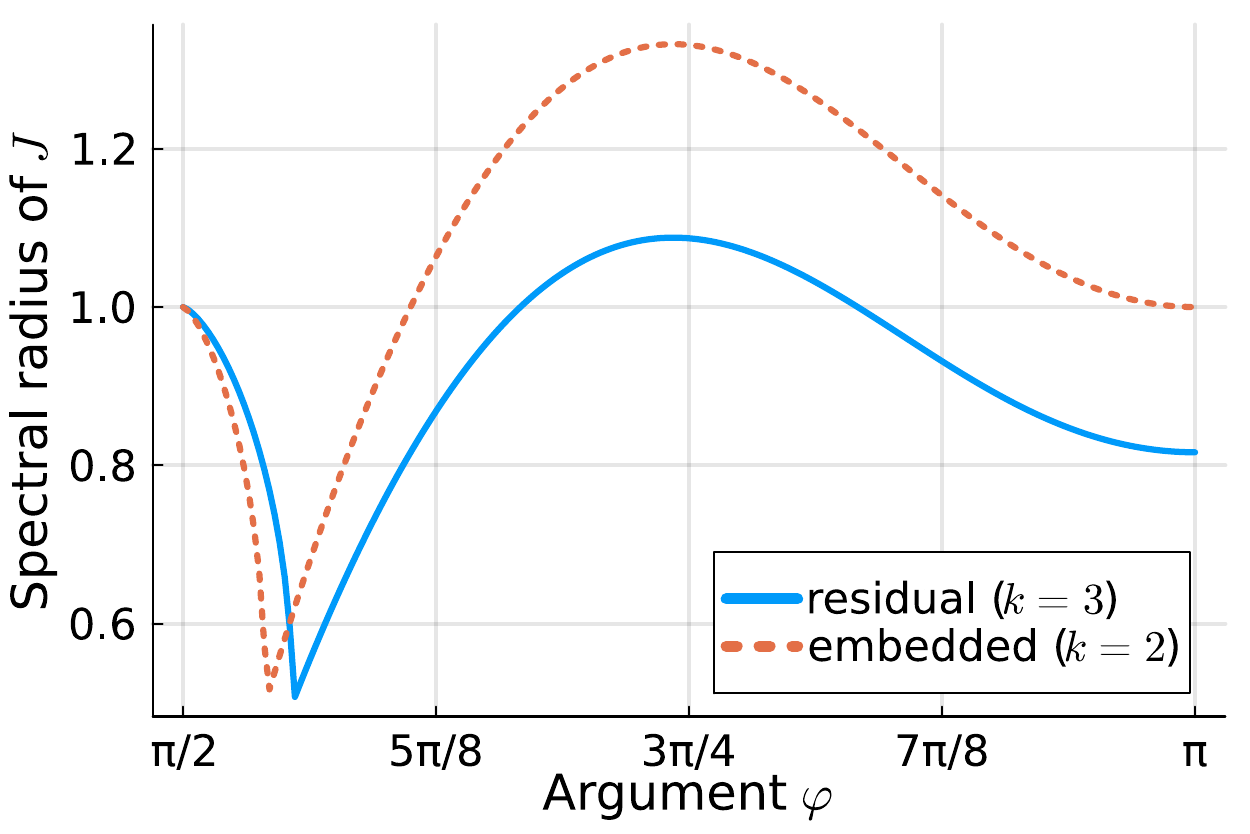}
    \caption{I controller with $\beta = (1, 0, 0)$,
             corresponding to the Jacobians
             \eqref{eq:jacobian-RK22-I-residual} and
             \eqref{eq:jacobian-RK22-I-embedded}.}
  \end{subfigure}%
  \hspace*{\fill}
  \begin{subfigure}{0.49\textwidth}
    \includegraphics[width=\textwidth]{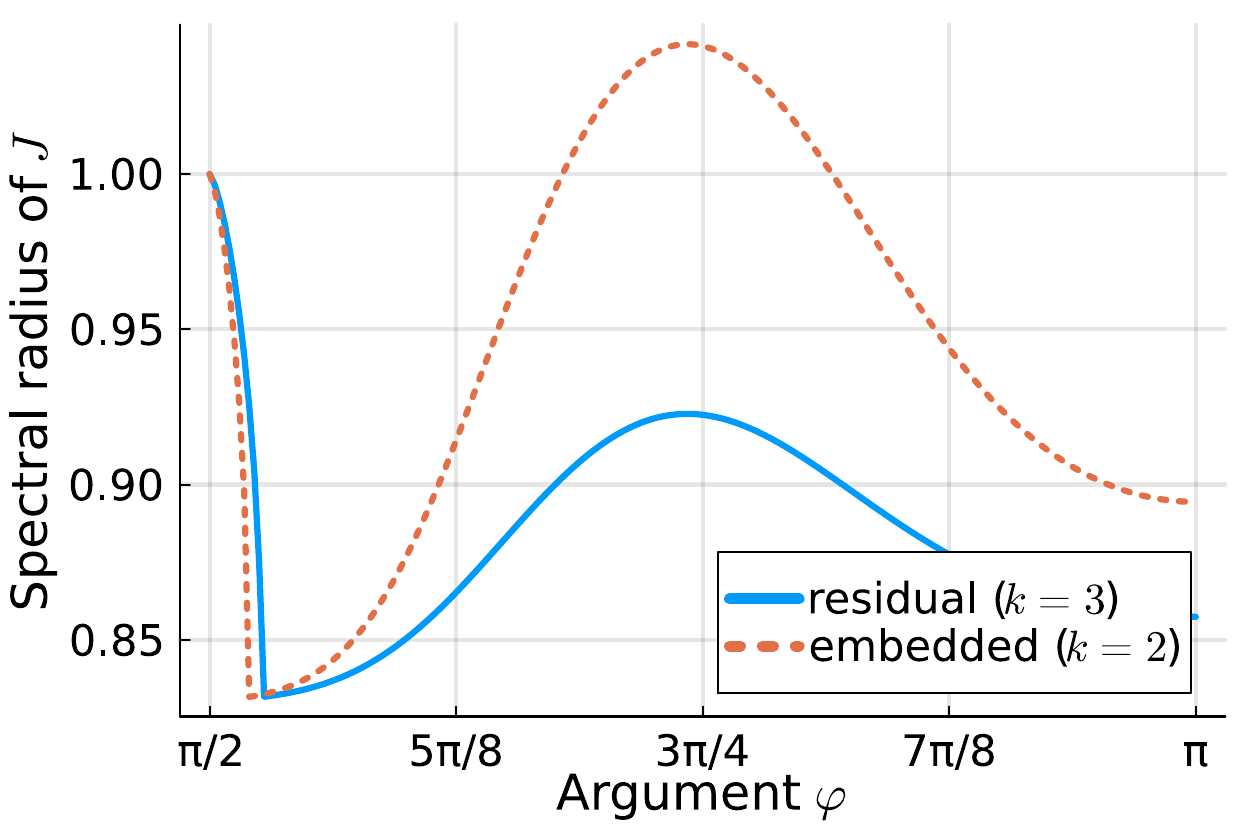}
    \caption{PI controller with $\beta = (0.6, -0.2, 0)$,
             corresponding to the Jacobians
             \eqref{eq:jacobian-RK22-PI-residual} and
             \eqref{eq:jacobian-PID-embedded}.}
  \end{subfigure}%
  \caption{Spectral radius of the Jacobian $J$ for error estimates based on
           the residual and an embedded Euler method for explicit second-order,
           two-stage Runge-Kutta methods. The Jacobian is evaluated at
           $z = r \e^{\i \phi}$ where the radius $r$ is chosen such that $z$
           is on the boundary of the stability region of the (main) method.}
  \label{fig:spectral_radius_rk22}
\end{figure}

The eigenvalues of these Jacobians are complex expressions that do not lend
themselves to an analytical investigation. Thus, we use a numerical approach
to compute and visualize the spectral radii in
Figure~\ref{fig:spectral_radius_rk22}.
First, it is clear that
the I controller does not lead to
step size control stability for all methods.
However, the instability is less severe for the residual-based approach ---
the spectral radius of the Jacobian is smaller in most regions and exceeds
unity less compared to the version using an embedded Euler method.

To check the behavior in practice, we use the test problem \eqref{eq:ode} with
\begin{equation}
\label{eq:test-problem-hairer-wanner}
  f(t, u) =
  -2000
  \begin{pmatrix}
    \cos(t) u_1 + \sin(t) u_2 + 1 \\
    -\sin(t) u_1 + \cos(t) u_2 + 1
  \end{pmatrix},
  \quad
  u^0 =
  \begin{pmatrix}
    1 \\ 0
  \end{pmatrix},
\end{equation}
in the time interval $(0.0, 1.57)$ as suggested by Hairer and Wanner
\cite[Section~IV.2]{hairer2010solving}.
With tolerances $\atol = \rtol = 10^{-4}$ and $k = 2$, the embedded approach
leads to 1877 accepted and 263 rejected steps. The $L^1$ residual-based approach
with $k = 3$
leads to 1811 accepted and 27 rejected steps.
The $L^2$ variant behaves similarly, see Table~\ref{tab:RK22}.
Details of the implementation and further numerical experiments are discussed
in Section~\ref{sec:numerical_experiments}.

The expression of the Jacobian \eqref{eq:jacobian-PID} of the PI controller
with residual error estimator becomes
\begin{equation}
\label{eq:jacobian-RK22-PI-residual}
  J =
  \begin{pmatrix}
    1 & \mu & 0 & 0 \\
    -\frac{\beta_1}{k} & 1 - \frac{3 \beta_1}{k} &
      -\frac{\beta_2}{k} & -\frac{3 \beta_2}{k} \\
    1 & 0 & 0 & 0 \\
    0 & 1 & 0 & 0
  \end{pmatrix},
  \quad
  \mu = \Re\biggl( \frac{z + z^2}{1 + z + z^2/2} \biggr).
\end{equation}
The spectral radii of the Jacobians for the PI controller with parameters
$\beta = (0.6, -0.2, 0.0)$ are
visualized in Figure~\ref{fig:spectral_radius_rk22}. Clearly, the more
involved controller leads to step size control stability for the
residual-based approach but not for the version using an embedded method.
For the test problem \eqref{eq:test-problem-hairer-wanner}, we get
1918 accepted and 55 rejected steps
for the the embedded approach while the residual-based approach leads to
1824 accepted and no rejected steps.

\begin{table}[htbp]
\centering
  \caption{Number of accepted and rejected time steps of
           Heun's second-order method with an embedded explicit Euler method
           for the test problem \eqref{eq:test-problem-hairer-wanner} with
           tolerances $\atol = \rtol = 10^{-4}$.}
  \label{tab:RK22}
  \begin{tabular*}{\linewidth}{@{\extracolsep{\fill}}c *3c *3c}
    \toprule
    & \multicolumn{3}{c}{I controller, $\beta = (1, 0, 0)$}
    & \multicolumn{3}{c}{PI controller, $\beta = (0.6, -0.2, 0)$}
    \\
    & \multicolumn{2}{c}{Residual Estimator}
    & \multicolumn{1}{c}{Embedded}
    & \multicolumn{2}{c}{Residual Estimator}
    & \multicolumn{1}{c}{Embedded}
    \\
    & \multicolumn{1}{c}{$L^1$}
    & \multicolumn{1}{c}{$L^2$}
    & \multicolumn{1}{c}{Method}
    & \multicolumn{1}{c}{$L^1$}
    & \multicolumn{1}{c}{$L^2$}
    & \multicolumn{1}{c}{Method}
    \\
    \midrule
    $k$ & 3 & 3 & 2
        & 3 & 3 & 2
    \\
    Accepted & 1811 & 1815 & 1877
              & 1824 & 1828 & 1918
    \\
    Rejected & 27 & 26 & 263
              & 0 & 0 & 55
    \\
    \bottomrule
  \end{tabular*}
\end{table}

\subsection{Third-order, three-stage methods}
\label{sec:BS3}

All explicit third-order, three-stage Runge-Kutta methods have the stability
function
\begin{equation}
  R(z) = 1 + z + \frac{z^2}{2} + \frac{z^3}{6}.
\end{equation}
We would like to use the central cubic Hermite interpolation polynomial
\begin{equation}
\label{eq:cubic-hermite-central}
\begin{aligned}
  \widehat{u}(t)
  &=
  \biggl( 1 - \frac{3 t^2}{\dt_n^2} + \frac{2 t^3}{\dt_n^3} \biggr) u^{n} +
  \biggl( t - \frac{2 t^2}{\dt_n} + \frac{t^3}{\dt_n^2} \biggr) f(t^{n}, u^{n})
  \\
  &\quad
  +
  \biggl( \frac{3 t^2}{\dt_n^2} - \frac{2 t^3}{\dt_n^3} \biggr) u^{n+1} +
  \biggl( -\frac{t^2}{\dt_n} + \frac{t^3}{\dt_n^2} \biggr) f(t^{n+1}, u^{n+1})
\end{aligned}
\end{equation}
normalized to $t \in [0, \dt_n]$.
However, we have not been able to evaluate the integral
$\| R \|_{L^1(t^n, t^{n+1})}$ analytically, even when using Mathematica
\cite{mathematica12}. Thus, we use the left-biased cubic Hermite interpolation
polynomial
\begin{equation}
\label{eq:cubic-hermite-left}
\begin{aligned}
  \widehat{u}(t)
  &=
  \biggl( 1 - \frac{t^3}{\dt_n^3} \biggr) u^{n} +
  \biggl( t - \frac{t^3}{\dt_n^2} \biggr) f(t^{n}, u^{n})
  \\
  &\quad
  +
  \biggl( \frac{t^2}{2} - \frac{t^3}{2 \dt_n} \biggr) (f_t + f_u f)(t^{n}, u^{n}) +
  \frac{t^3}{\dt_n^3} u^{n+1}
\end{aligned}
\end{equation}
normalized to $t \in [0, \dt_n]$
for a first analysis but the central version in the implementation and a
more numerically supported analysis.

\subsubsection{Left-biased cubic Hermite interpolation}

The analysis proceeds with the weighted $L^1$ error estimate
\eqref{eq:weighted-estimate-residual-L1} given by
\begin{equation}
\begin{aligned}
  |e_{n+1}|
  =
  \| R \|_{L^1(t^n, t^{n+1})}
  &=
  \int_{t^{n}}^{t^{n+1}} \left|
    \frac{\dif}{\dif t} \widehat{u}(t) - \lambda \widehat{u}(t)
  \right| \dif t
  \\
  &=
  \frac{1}{6} \int_{0}^{\dt_n} t^3 \dif t \;
  |\lambda|^4 |u^{n}|
  =
  \frac{1}{24} \dt_n^4 |\lambda|^4 |u^{n}|.
\end{aligned}
\end{equation}
The expression of the Jacobian \eqref{eq:jacobian-I} of the I controller with
residual error estimator becomes
\begin{equation}
\label{eq:jacobian-BS3-I-residual}
  J =
  \begin{pmatrix}
    1 & \mu \\
  -\frac{1}{k} & 1 - \frac{4}{k} \\
  \end{pmatrix},
  \quad
  \mu = \Re\biggl( \frac{z + z^2 + z^3 / 2}{1 + z + z^2/2 + z^3 / 6} \biggr).
\end{equation}
If we use the $L^2$ error estimate \eqref{eq:weighted-estimate-residual-L2}
instead, we get
\begin{equation}
  |e_{n+1}|
  =
  \sqrt{\dt_n} \| R \|_{L^2(t^n, t^{n+1})}
  =
  \frac{1}{6 \sqrt{7}} \dt_n^4 |\lambda|^4 |u^{n}|
\end{equation}
for the left-biased cubic Hermite interpolation
\eqref{eq:cubic-hermite-left} and thus the same Jacobian as for the $L^1$
error estimate discussed above.

\begin{figure}[htb]
\centering
  \begin{subfigure}{0.49\textwidth}
    \includegraphics[width=\textwidth]{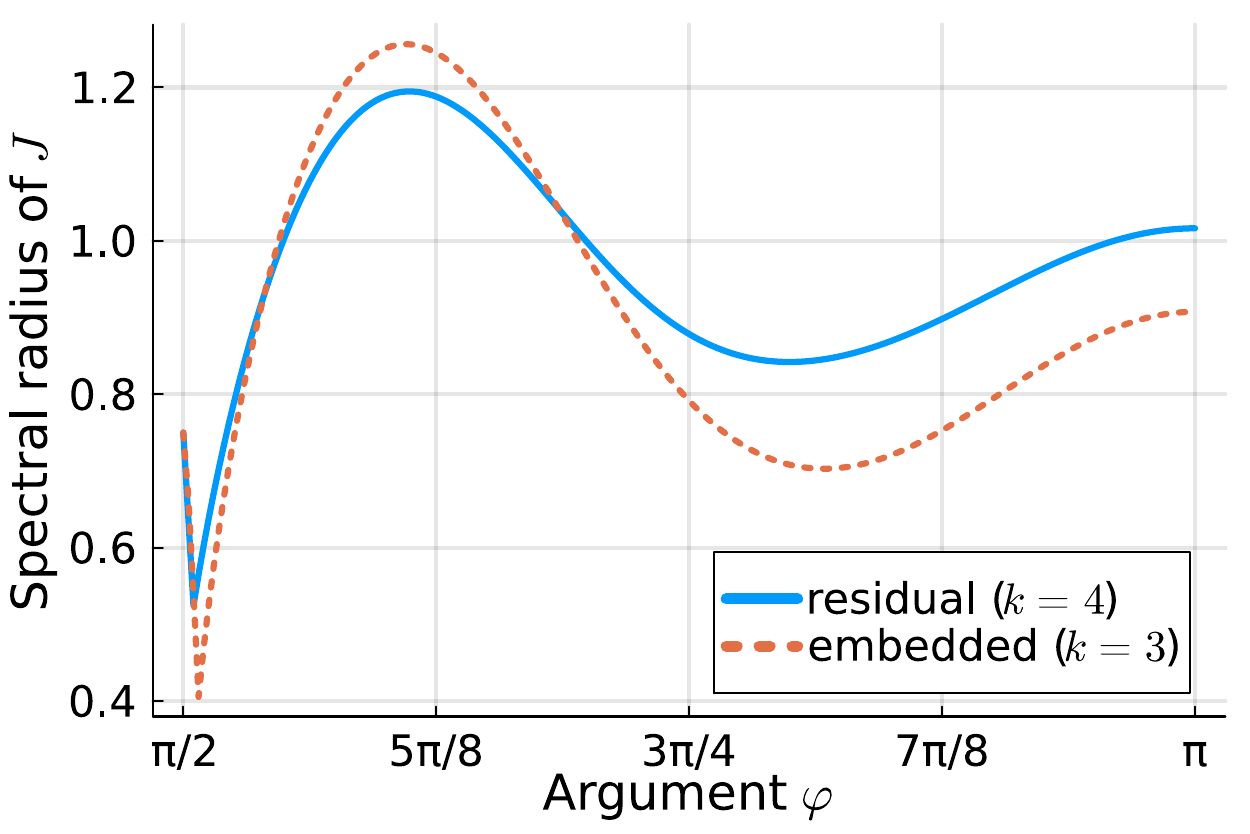}
    \caption{I controller with $\beta = (1, 0, 0)$,
             corresponding to the Jacobians
             \eqref{eq:jacobian-BS3-I-residual} and
             \eqref{eq:jacobian-PID-embedded}.}
  \end{subfigure}%
  \hspace*{\fill}
  \begin{subfigure}{0.49\textwidth}
    \includegraphics[width=\textwidth]{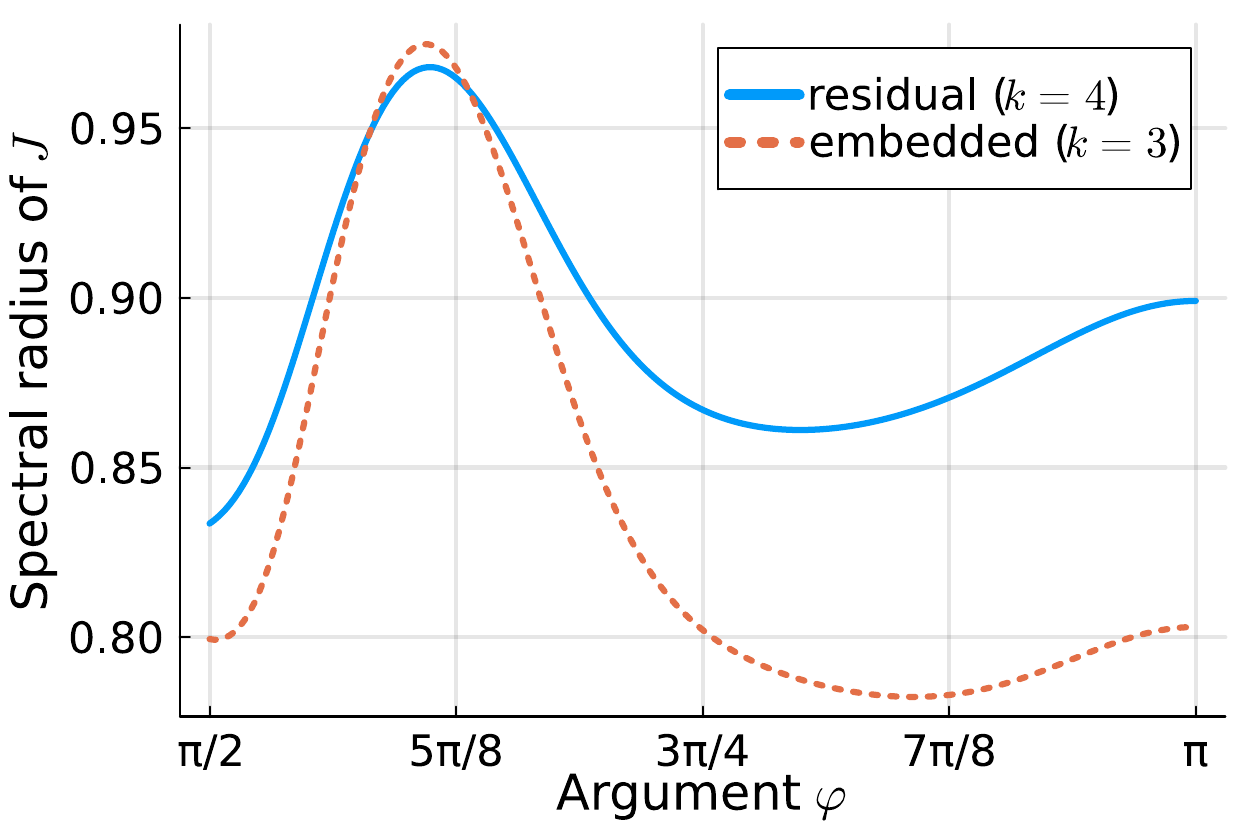}
    \caption{PI controller with $\beta = (0.6, -0.2, 0)$,
             corresponding to the Jacobians
             \eqref{eq:jacobian-BS3-PI-residual} and
             \eqref{eq:jacobian-PID-embedded}.}
  \end{subfigure}%
  \caption{Spectral radius of the Jacobian $J$ for error estimates based on
           the $L^1$ residual and the second-order embedded method for the
           third-order method of Bogacki and Shampine \cite{bogacki1989a32}
           with left-biased cubic Hermite interpolation.
           The Jacobian is evaluated at $z = r \e^{\i \phi}$ where the radius
           $r$ is chosen such that $z$ is on the boundary of the stability
           region of the (main) method.}
  \label{fig:spectral_radius_bs3}
\end{figure}

The spectral radii of the Jacobians are visualized in
Figure~\ref{fig:spectral_radius_bs3}.
Again, all methods do not lead to step size control stability.

The expression of the Jacobian \eqref{eq:jacobian-PID} of the PI controller
with residual error estimator becomes
\begin{equation}
\label{eq:jacobian-BS3-PI-residual}
  J =
  \begin{pmatrix}
    1 & \mu & 0 & 0 \\
    -\frac{\beta_1}{k} & 1 - \frac{4 \beta_1}{k} &
      -\frac{\beta_2}{k} & -\frac{4 \beta_2}{k} \\
    1 & 0 & 0 & 0 \\
    0 & 1 & 0 & 0
  \end{pmatrix},
  \quad
  \mu = \Re\biggl( \frac{z + z^2 + z^3 / 2}{1 + z + z^2/2 + z^3 / 6} \biggr).
\end{equation}
The spectral radii of the Jacobians for the PI controller with parameters
$\beta = (0.6, -0.2, 0.0)$ recommended in \cite{ranocha2021optimized} are
visualized in Figure~\ref{fig:spectral_radius_bs3}. Clearly, the more
involved controller leads to step size control stability for all approaches.

\subsubsection{Central cubic Hermite interpolation}

Recall the Jacobian \eqref{eq:jacobian-I} of the I controller system.
For the $L^1$ residual error estimator with central cubic Hermite
interpolation, the Jacobian is
\begin{equation}
\label{eq:jacobian-BS3-I-residual-central}
  J =
  \begin{pmatrix}
    1 & \mu \\
    -\frac{1}{k} &
      1 - \frac{1}{k} \partial_{\chi_{n}} \log|e_{n+1}|
  \end{pmatrix},
  \qquad
  \mu = \Re\biggl( \frac{R'(z)}{R(z)} z \biggr),
\end{equation}
where $\eta_n = \log|u^n|$, $\chi_n = \log \dt_n$, and $z = \lambda \dt_n$.
We have not been able to compute the expression $\partial_{\chi_n} |e_{n+1}|$
analytically in this case. However, we can evaluate it numerically by using
an adaptive Gauss-Kronrod quadrature with relative tolerance $10^{-8}$ and
absolute tolerance $10^{-14}$ implemented in QuadGK.jl \cite{johnson2013quadgk}
for the integrals
\begin{equation}
\begin{aligned}
  |e_{n+1}|
  &=
  \int_{0}^{\dt_n} h(t) \dif t,
  \qquad
  h(t) := \frac{1}{6} t (\dt_n - t) \bigl|t - 2 \dt_n + t \dt_n \lambda\bigr| |\lambda|^4 |u^n|,
  \\
  \partial_{\chi_n} |e_{n+1}|
  &=
  \left(
    h(\dt_n) + \int_{0}^{\dt_n} \frac{\partial h(t)}{\partial \dt_n} \dif t
  \right) \dt_n,
\end{aligned}
\end{equation}
where the derivatives are evaluated using ForwardDiff.jl
\cite{revels2016forward}.

\begin{figure}[htb]
\centering
  \begin{subfigure}{0.49\textwidth}
    \includegraphics[width=\textwidth]{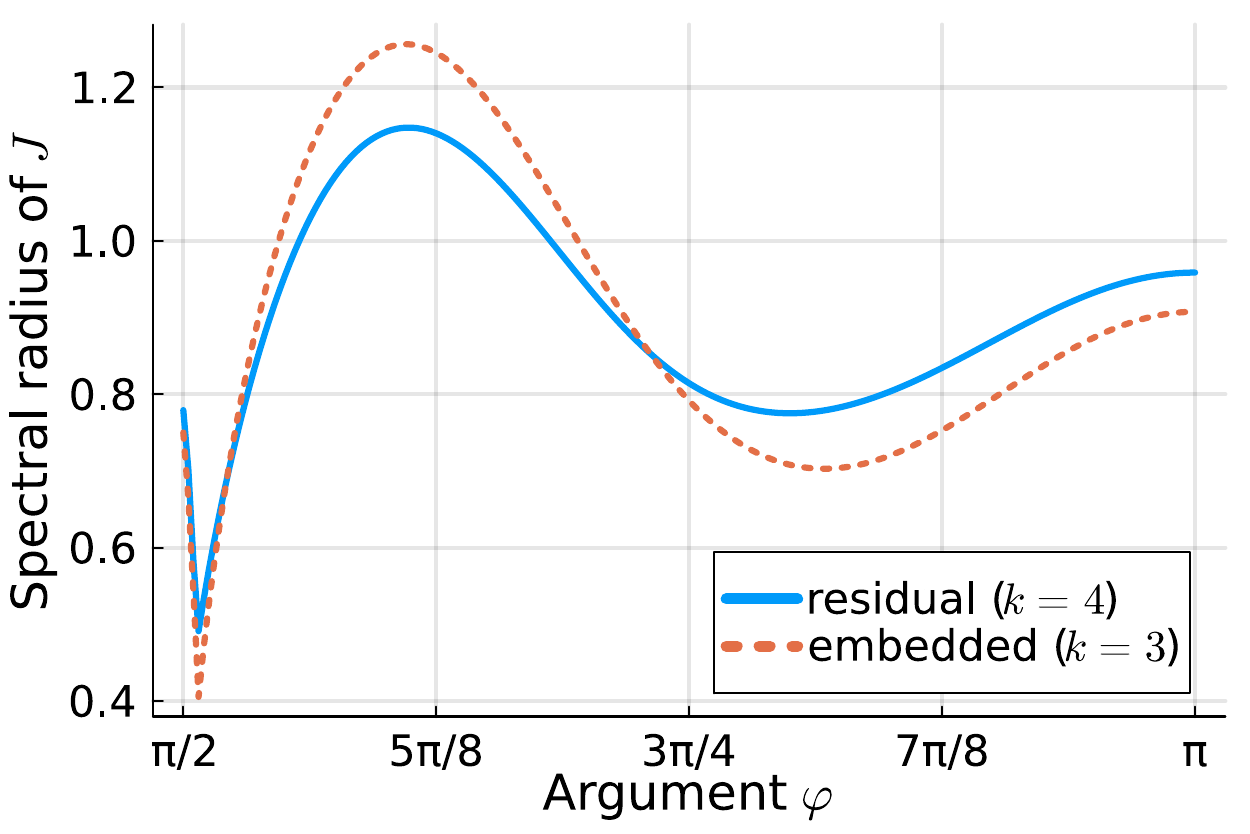}
    \caption{I controller with $\beta = (1, 0, 0)$,
             corresponding to the Jacobians
             \eqref{eq:jacobian-BS3-I-residual-central} and
             \eqref{eq:jacobian-PID-embedded}.}
  \end{subfigure}%
  \hspace*{\fill}
  \begin{subfigure}{0.49\textwidth}
    \includegraphics[width=\textwidth]{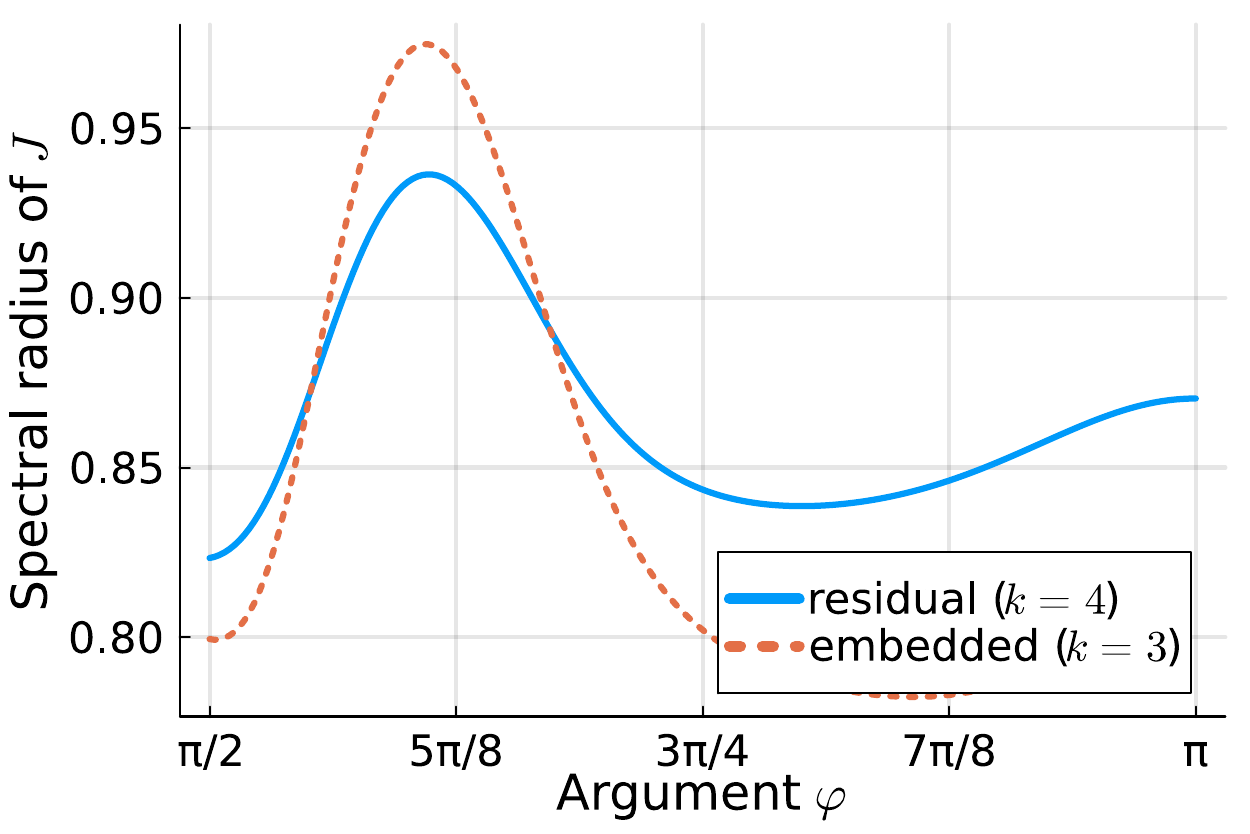}
    \caption{PI controller with $\beta = (0.6, -0.2, 0)$,
             corresponding to the Jacobians
             \eqref{eq:jacobian-BS3-PI-residual-central} and
             \eqref{eq:jacobian-PID-embedded}.}
  \end{subfigure}%
  \caption{Spectral radius of the Jacobian $J$ for error estimates based on
           the $L^1$ residual and the second-order embedded method for the
           third-order method of Bogacki and Shampine \cite{bogacki1989a32}
           with central cubic Hermite interpolation.
           The Jacobian is evaluated at $z = r \e^{\i \phi}$ where the radius
           $r$ is chosen such that $z$ is on the boundary of the stability
           region of the (main) method.}
  \label{fig:spectral_radius_bs3_quadrature}
\end{figure}

The spectral radii of the Jacobians are visualized in
Figure~\ref{fig:spectral_radius_bs3_quadrature}.
Again, both methods do not lead to step size control stability for the simple
I controller.
It is interesting to see that the general trend of the spectral radius
is similar to the one computed for the left-biased cubic Hermite interpolation
shown in Figure~\ref{fig:spectral_radius_bs3}. However, the spectral radius does
not exceed unity near $\phi = \pi$ for the central interpolation (while still
being close to unity). This is in accordance with numerical results presented
later in Section~\ref{sec:krogh}.

Using the same test problem \eqref{eq:test-problem-hairer-wanner} as for
second-order, two-stage methods results in
1318 accepted and 120 rejected steps
for the the embedded approach while the residual-based approach leads to
1327 accepted and 21 rejected steps,
see Table~\ref{tab:BS3}.
This is in accordance with the spectral radius of the residual-based approach
exceeding unity less than the embedded approach.

The expression of the Jacobian \eqref{eq:jacobian-PID} of the PI controller
with residual error estimator becomes
\begin{equation}
\label{eq:jacobian-BS3-PI-residual-central}
  J =
  \begin{pmatrix}
    1 & \mu & 0 & 0 \\
    -\frac{\beta_1}{k} & 1 - \frac{\beta_1}{k} \frac{\partial |e_{n+1}|}{\partial \chi_{n}}  &
      -\frac{\beta_2}{k} & -\frac{\beta_2}{k} \frac{\partial |e_{n}|}{\partial \chi_{n-1}}  \\
    1 & 0 & 0 & 0 \\
    0 & 1 & 0 & 0
  \end{pmatrix},
  \quad
  \mu = \Re\biggl( \frac{z + z^2 + z^3 / 2}{1 + z + z^2/2 + z^3 / 6} \biggr).
\end{equation}
The spectral radii of the Jacobians for the PI controller with parameters
$\beta = (0.6, -0.2, 0.0)$ recommended in \cite{ranocha2021optimized} are
visualized in Figure~\ref{fig:spectral_radius_bs3_quadrature}. Clearly, the
more involved controller leads to step size control stability for all
approaches.
For the test problem \eqref{eq:test-problem-hairer-wanner}, we get
1330 accepted and 1 rejected steps
for the the embedded approach while the residual-based approach leads to
1333 accepted and 1 rejected steps.

Finally, we are able to compute the $L^2$ error estimate for the
central cubic Hermite interpolation \eqref{eq:cubic-hermite-central}
analytically, resulting in
\begin{equation}
  |e_{n+1}|
  =
  \sqrt{\dt_n} \| R \|_{L^2(t^n, t^{n+1})}
  =
  \frac{1}{6 \sqrt{105}} \dt_n^4 |\lambda|^4 |u^{n}|
    \sqrt{8 + \dt^2 |\lambda|^2 - 5 \dt_n \Re(\lambda)}.
\end{equation}
Then, the expression of the Jacobian \eqref{eq:jacobian-I} of the I controller
with $L^2$ residual error estimator becomes
\begin{equation}
\label{eq:jacobian-BS3-I-residual-central-L2}
\begin{gathered}
  J =
  \begin{pmatrix}
    1 & \mu \\
  -\frac{1}{k} & J_{22} \\
  \end{pmatrix},
  \\
  \mu = \Re\biggl( \frac{z + z^2 + z^3 / 2}{1 + z + z^2/2 + z^3 / 6} \biggr),
  \;
  J_{22} = 1 - \frac{64 + 10 |z|^2 - 45 \Re(z)}{2 k \bigl( 8 + |z|^2 - 5 \Re(z) \bigr)}.
\end{gathered}
\end{equation}

\begin{figure}[htb]
\centering
  \begin{subfigure}{0.49\textwidth}
    \includegraphics[width=\textwidth]{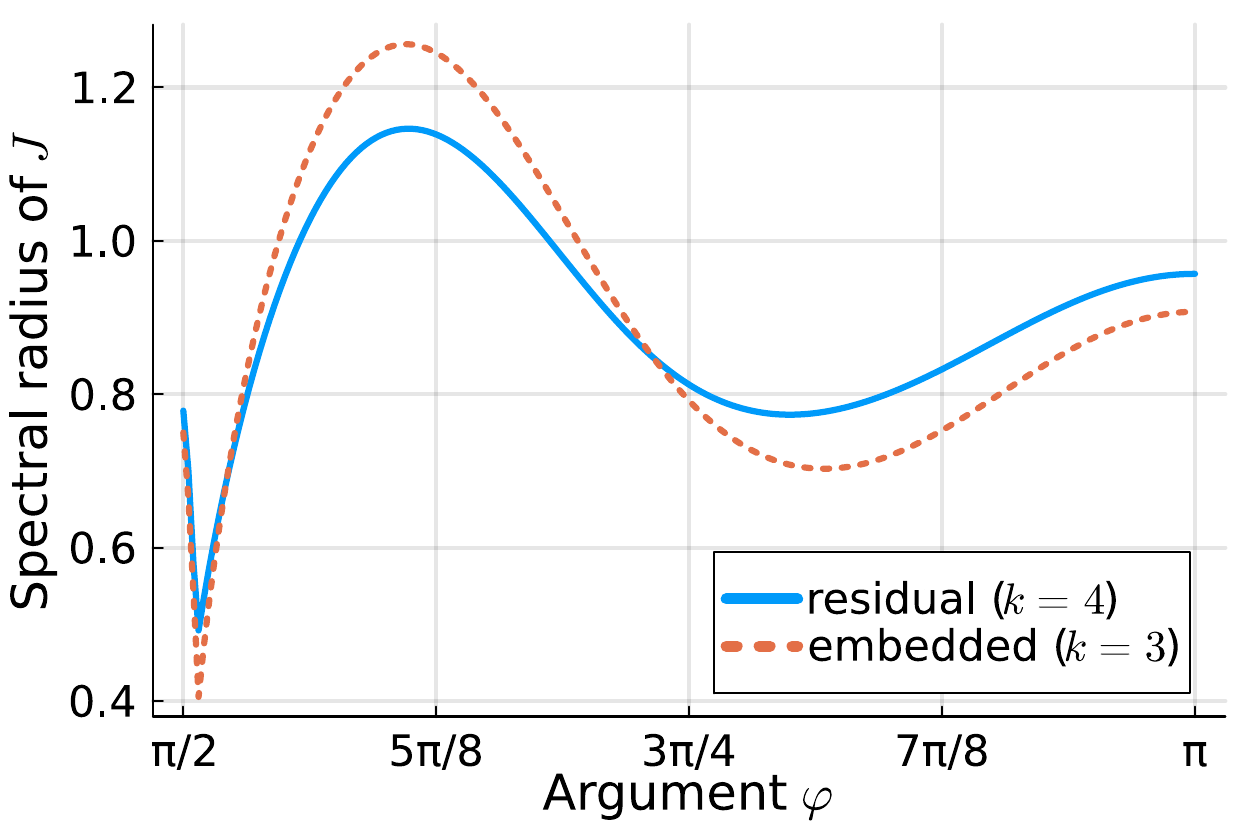}
    \caption{I controller with $\beta = (1, 0, 0)$,
             corresponding to the Jacobians
             \eqref{eq:jacobian-BS3-I-residual-central-L2} and
             \eqref{eq:jacobian-PID-embedded}.}
  \end{subfigure}%
  \hspace*{\fill}
  \begin{subfigure}{0.49\textwidth}
    \includegraphics[width=\textwidth]{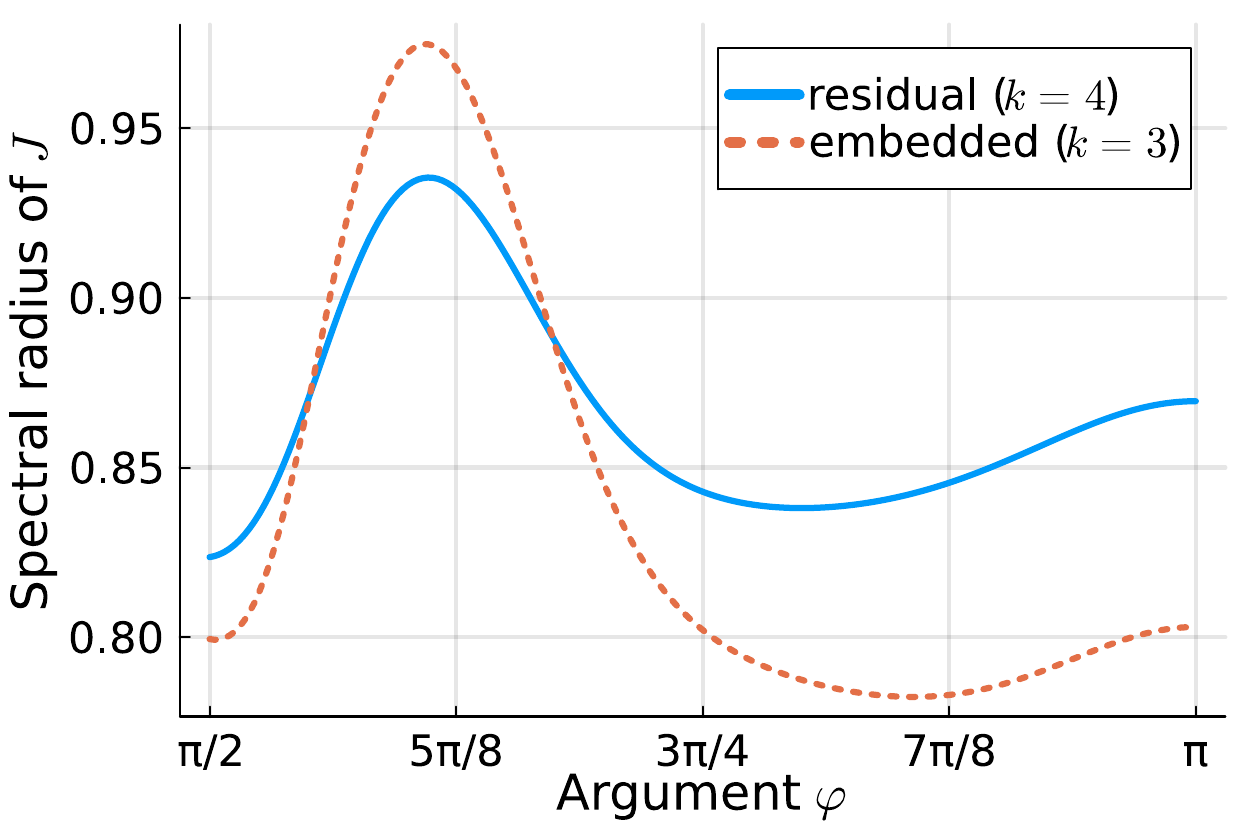}
    \caption{PI controller with $\beta = (0.6, -0.2, 0)$,
             corresponding to the Jacobians
             \eqref{eq:jacobian-BS3-PI-residual-central-L2} and
             \eqref{eq:jacobian-PID-embedded}.}
  \end{subfigure}%
  \caption{Spectral radius of the Jacobian $J$ for error estimates based on
           the $L^2$ residual and the second-order embedded method for the
           third-order method of Bogacki and Shampine \cite{bogacki1989a32}
           with central cubic Hermite interpolation.
           The Jacobian is evaluated at $z = r \e^{\i \phi}$ where the radius
           $r$ is chosen such that $z$ is on the boundary of the stability
           region of the (main) method.}
  \label{fig:spectral_radius_bs3_l2}
\end{figure}

The spectral radii of the Jacobians are visualized in
Figure~\ref{fig:spectral_radius_bs3_l2}. The spectral radii of the
$L^2$ error estimator with central cubic Hermite interpolation
\eqref{eq:cubic-hermite-central} exceed unity less than the spectral radii
of the $L^1$ error estimator \eqref{eq:weighted-estimate-residual-L1} with
left-biased cubic Hermite interpolation \eqref{eq:cubic-hermite-left}. However,
they still do not lead to step size control stability.

The expression of the Jacobian \eqref{eq:jacobian-PID} of the PI controller
with $L^2$ residual error estimator becomes
\begin{equation}
\label{eq:jacobian-BS3-PI-residual-central-L2}
\begin{gathered}
  J =
  \begin{pmatrix}
    1 & \mu & 0 & 0 \\
    -\frac{\beta_1}{k} & 1 - \beta_1 \alpha &
      -\frac{\beta_2}{k} & -\beta_2 \alpha \\
    1 & 0 & 0 & 0 \\
    0 & 1 & 0 & 0
  \end{pmatrix},
  \\
  \mu = \Re\biggl( \frac{z + z^2 + z^3 / 2}{1 + z + z^2/2 + z^3 / 6} \biggr),
  \;
  \alpha = \frac{64 + 10 |z|^2 - 45 \Re(z)}{2 k (8 + |z|^2 - 5 \Re(z)}.
\end{gathered}
\end{equation}
The spectral radii of the Jacobians for the PI controller with parameters
$\beta = (0.6, -0.2, 0.0)$ recommended in \cite{ranocha2021optimized} are
visualized in Figure~\ref{fig:spectral_radius_bs3_l2}. Clearly, the more
involved controller leads to step size control stability for all approaches.
Again, the $L^2$ error estimator with central cubic Hermite interpolation
leads to more damping around $\phi = 5 \pi / 8$ than the $L^1$ version with
left-biased interpolation.

Using the $L^2$ error estimator with central cubic Hermite interpolation
but otherwise the same setup as before, we get
1326 accepted and 25 rejected steps for the I controller and
1333 accepted and 1 rejected steps for the PI controller,
see Table~\ref{tab:BS3}.

\begin{table}[htbp]
\centering
  \caption{Number of accepted and rejected time steps of
           the Runge-Kutta pair of Bogacki and Shampine \cite{bogacki1989a32}
           for the test problem \eqref{eq:test-problem-hairer-wanner} with
           tolerances $\atol = \rtol = 10^{-4}$ using the central cubic
           Hermite interpolation for the residual error estimators.}
  \label{tab:BS3}
  \begin{tabular*}{\linewidth}{@{\extracolsep{\fill}}c *3c *3c}
    \toprule
    & \multicolumn{3}{c}{I controller, $\beta = (1, 0, 0)$}
    & \multicolumn{3}{c}{PI controller, $\beta = (0.6, -0.2, 0)$}
    \\
    & \multicolumn{2}{c}{Residual Estimator}
    & \multicolumn{1}{c}{Embedded}
    & \multicolumn{2}{c}{Residual Estimator}
    & \multicolumn{1}{c}{Embedded}
    \\
    & \multicolumn{1}{c}{$L^1$}
    & \multicolumn{1}{c}{$L^2$}
    & \multicolumn{1}{c}{Method}
    & \multicolumn{1}{c}{$L^1$}
    & \multicolumn{1}{c}{$L^2$}
    & \multicolumn{1}{c}{Method}
    \\
    \midrule
    $k$ & 4 & 4 & 3
        & 4 & 4 & 3
    \\
    Accepted & 1327 & 1326 & 1318
             & 1333 & 1333 & 1330
    \\
    Rejected & 21 & 25 & 120
             & 1 & 1 & 1
    \\
    \bottomrule
  \end{tabular*}
\end{table}

\section{Numerical experiments}
\label{sec:numerical_experiments}

We have implemented all methods in Julia \cite{bezanson2017julia} and
use OrdinaryDiffEq.jl \cite{rackauckas2017differentialequations} for
the classical schemes with embedded methods. We use an adaptive Gauss-Kronrod
quadrature implemented in QuadGK.jl \cite{johnson2013quadgk} to compute the
integrals appearing in the residual error estimators; we set the relative
error tolerance to $10^{-8}$ for the ODE tests and to $10^{-6}$ for
the tests involving partial differential equations (PDEs).
The absolute and relative tolerances of the step size controller are equal,
$\atol = \rtol = \tol$. We use an adaptation of the algorithm of
\cite[Section~II.4]{hairer2008solving} to determine the initial time step size.
We use FFTW.jl \cite{frigo2005design} via the interface provided by
SummationByPartsOperators.jl \cite{ranocha2021sbp} for Fourier collocation
methods and Trixi.jl \cite{ranocha2022adaptive,schlottkelakemper2021purely}
for discontinuous Galerkin discretizations of conservation laws.
Finally, we use Plots.jl \cite{christ2023plots} to visualize the results.
We use \EPS control for all experiments.
All source code required to reproduce the numerical experiments is available
in our reproducibility repository \cite{ranocha2023stabilityRepro}.

First, we test the step size control stability theory with a nonlinear
ODE in Section~\ref{sec:krogh}. Thereafter, we study the methods in the two
regimes important for step size control of explicit time integration schemes:
the asymptotic regime of small time step sizes and the stability-limited regime.
We choose a classical ODE problem (Section~\ref{sec:rigidbody}) and a
non-stiff PDE (Section~\ref{sec:bbm}) for the asymptotic regime. Afterwards,
we consider hyperbolic conservation laws to study the stability-limited regime
in Sections~\ref{sec:linadv} and \ref{sec:euler}.

In all cases, we just show numerical results for the $L^1$ residual error
estimates. The corresponding results based on $L^2$ error estimates are
very similar (and can also be reproduced using the code of our
reproducibility repository \cite{ranocha2023stabilityRepro}).

\subsection{A nonlinear ODE}
\label{sec:krogh}

First, we follow \cite{higham1990embedded} and consider the nonlinear test
problem
\begin{equation}
\label{eq:krogh}
  u' = -B u + U^T (z_1^2 /2 - z_2^2 / 2, z_1 z_2, z_3^2, z_4^2)^T,
  \quad
  u(0) = (0, -2, -1, -1)^T,
\end{equation}
of Krogh \cite{krogh1973testing} with
\begin{equation}
  z = U u,
  \quad
  B = U^T \begin{pmatrix}
            -10 \cos(\phi) & - 10 \sin(\phi) & 0 & 0 \\
            10 \sin(\phi) & -10 \cos(\phi) & 0 & 0 \\
            0 & 0 & 1 & 0 \\
            0 & 0 & 0 & 1 / 2
          \end{pmatrix} U,
\end{equation}
where $U \in \R^{4 \times 4}$ contains $-1/2$ on the diagonal and $+1 / 2$
in all other components. For a fixed parameter $\phi$, the dominant eigenvalues
of the Jacobian for $t \to \infty$ become
$-10 |\cos(\phi)| \pm \i 10 \sin(\phi)$.

\begin{figure}[htbp]
\centering
  \begin{subfigure}{0.49\textwidth}
  \centering
   \includegraphics[width=\textwidth]{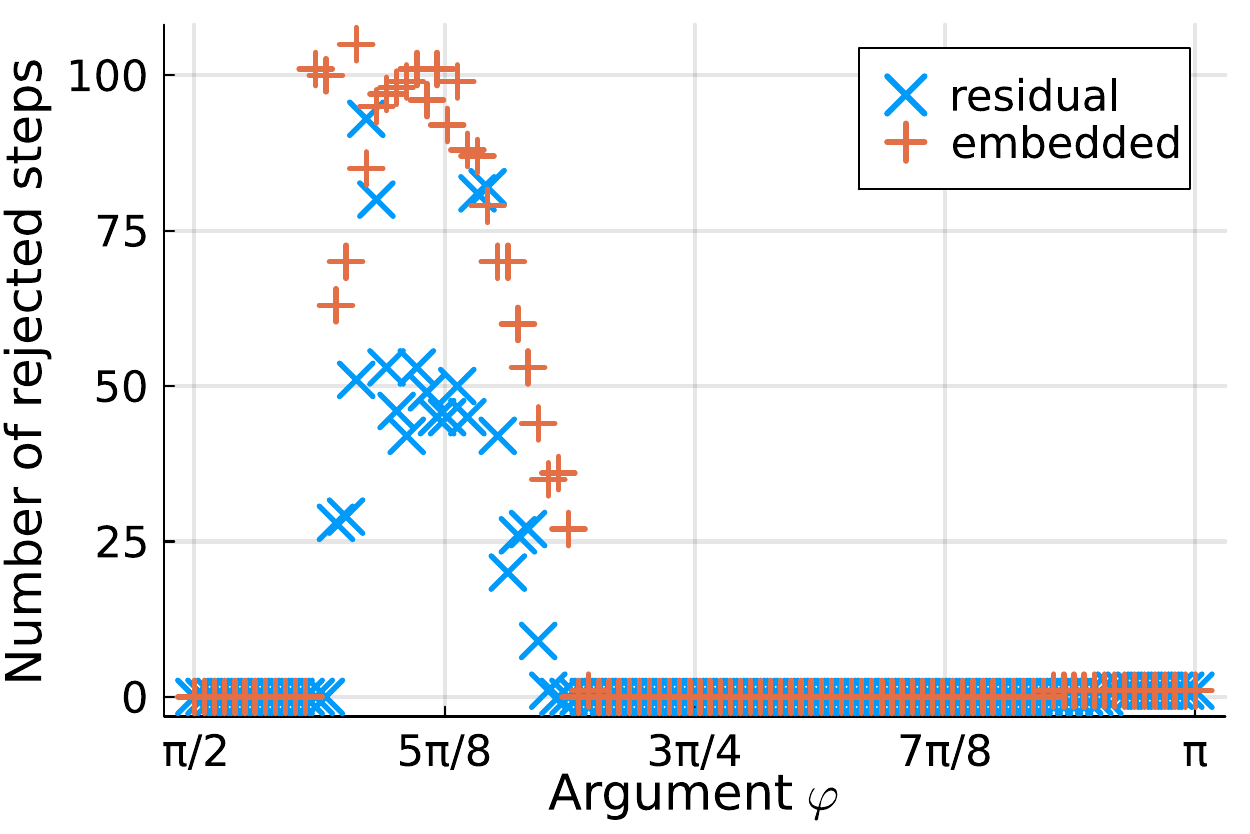}
   \caption{I controller, $\beta = (1, 0)$.}
  \end{subfigure}%
  \vspace*{\fill}
  \begin{subfigure}{0.49\textwidth}
  \centering
   \includegraphics[width=\textwidth]{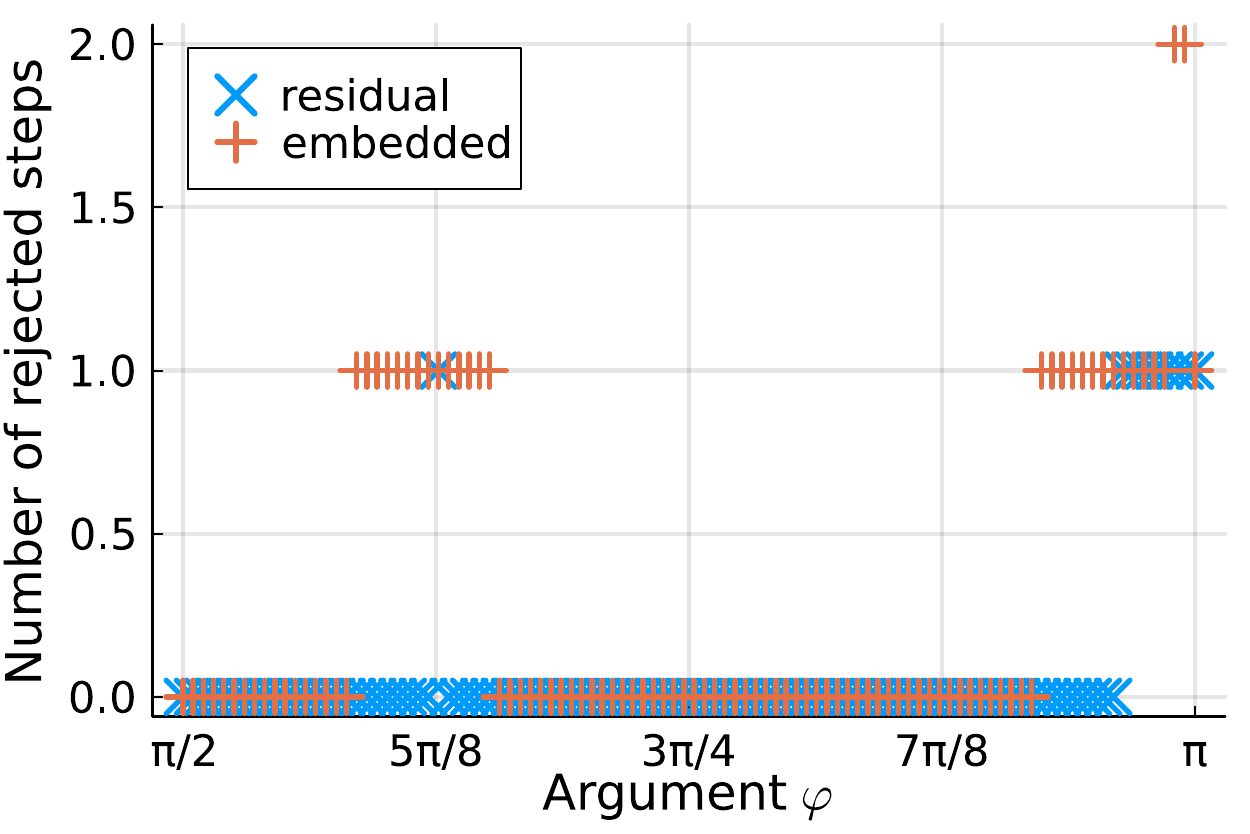}
   \caption{PI controller, $\beta = (0.6, -0.2)$.}
  \end{subfigure}%
  \caption{Number of rejected steps for the nonlinear ODE \eqref{eq:krogh}
           using the the third-order method of
           Bogacki and Shampine \cite{bogacki1989a32} with
           tolerance $\tol = 10^{-4}$.}
  \label{fig:krogh}
\end{figure}

We integrate the problem in the time interval $[0, 100]$ using the
third-order method of Bogacki and Shampine \cite{bogacki1989a32}
with both the embedded method and the $L^1$ residual error estimator.
The number of rejected steps for the simple I controller and the
PI controller given by $\beta = (0.6, -0.2)$ are shown in
Figure~\ref{fig:krogh}. Clearly, a significant number of steps is
rejected when the parameter $\phi$ is in the region around $\phi = 5 \pi / 8$
where the simple I controller is not stable. In contrast, the advanced PI
controller leads to at most one or two rejected steps for all values
of $\phi$.

\subsection{Euler equations of a rigid body}
\label{sec:rigidbody}

We consider the Euler equations of a rigid body with parameters used by
Krogh \cite{krogh1973testing}, i.e.,
\begin{equation}
  u'(t) = \begin{pmatrix} u_2 u_3 \\ -u_1 u_3 \\ -0.51 u_1 u_2 \end{pmatrix},
  \quad
  u(0) = \begin{pmatrix} 0 \\ 1 \\ 1 \end{pmatrix}.
\end{equation}
The solution is periodic with periodic given by Krogh \cite{krogh1973testing}.
We compute the $\ell^2$ error after one period.

\begin{figure}[htb]
\centering
  \begin{subfigure}{0.49\textwidth}
  \centering
   \includegraphics[width=\textwidth]{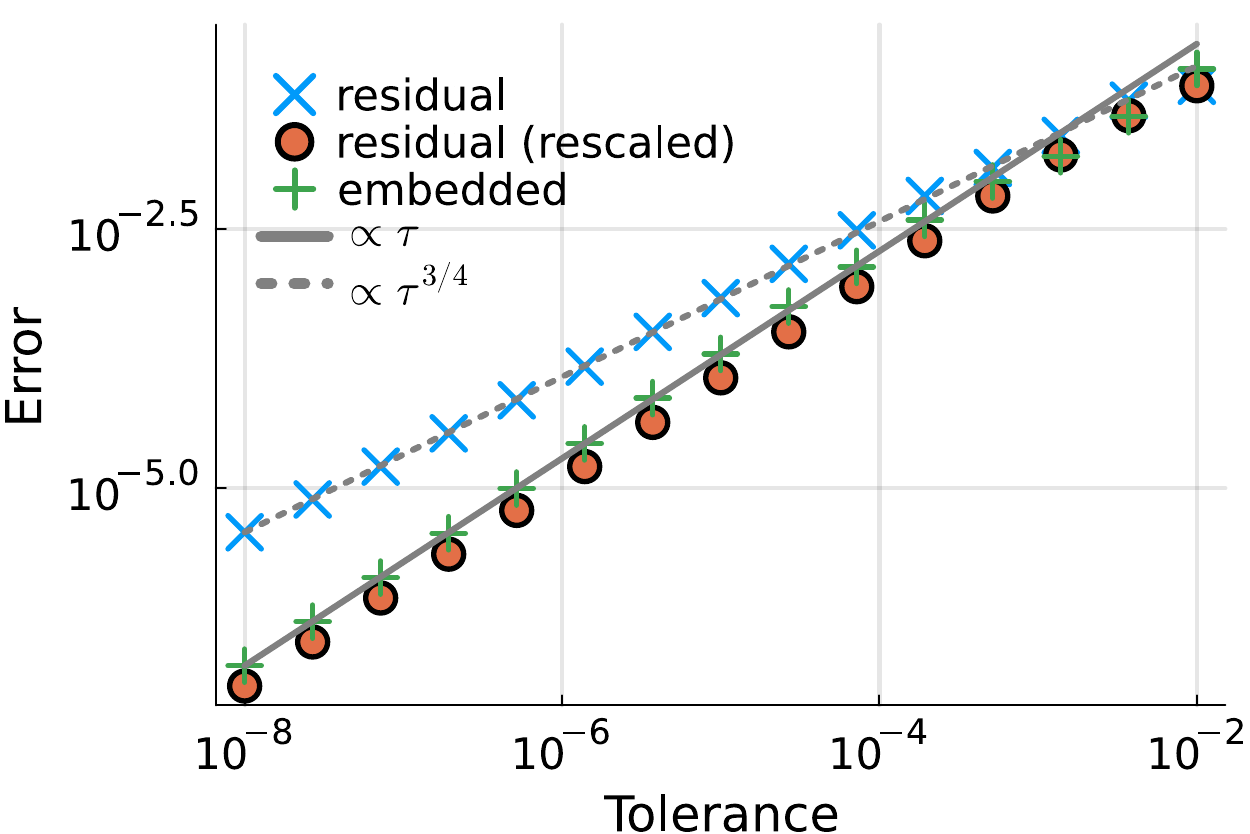}
    \caption{Error for different tolerances.}
  \end{subfigure}%
  \hspace*{\fill}
  \begin{subfigure}{0.49\textwidth}
  \centering
   \includegraphics[width=\textwidth]{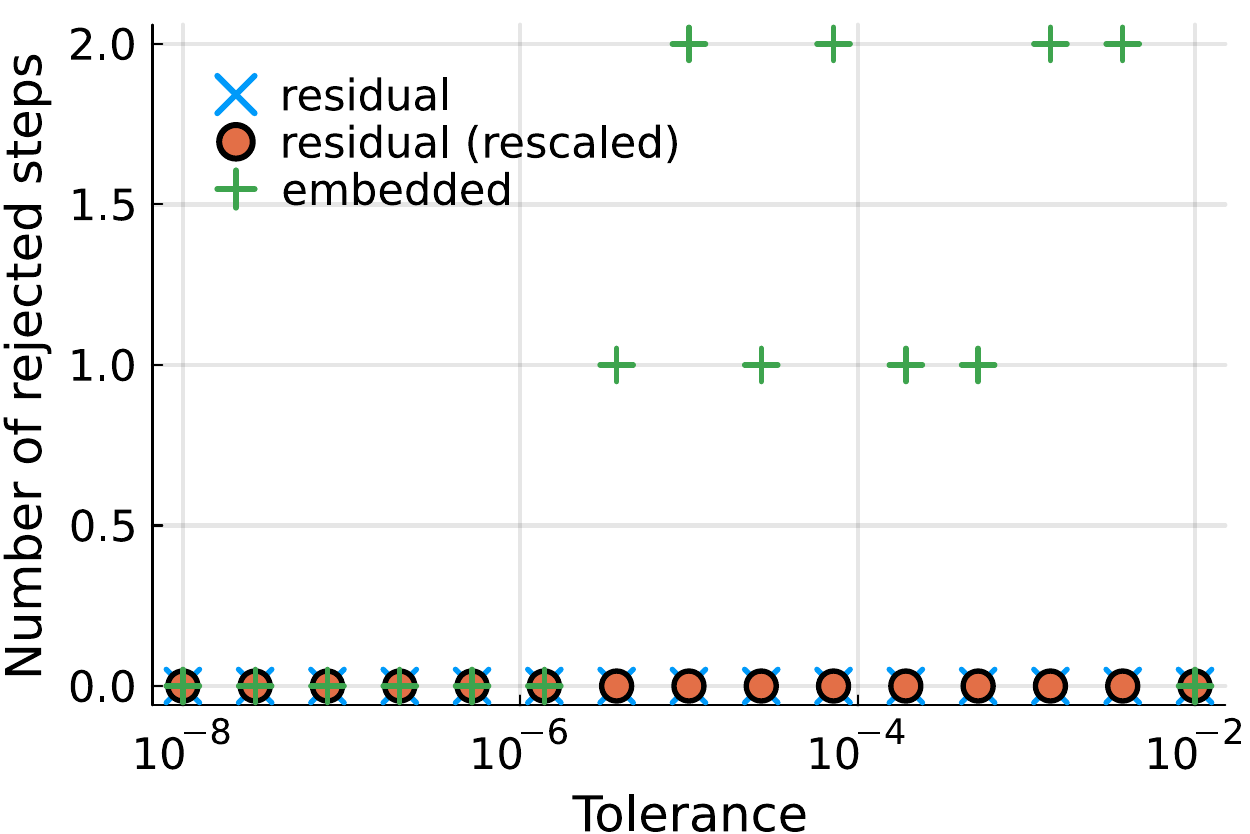}
    \caption{Number of rejected steps.}
  \end{subfigure}%
  \\
  \begin{subfigure}{0.49\textwidth}
  \centering
   \includegraphics[width=\textwidth]{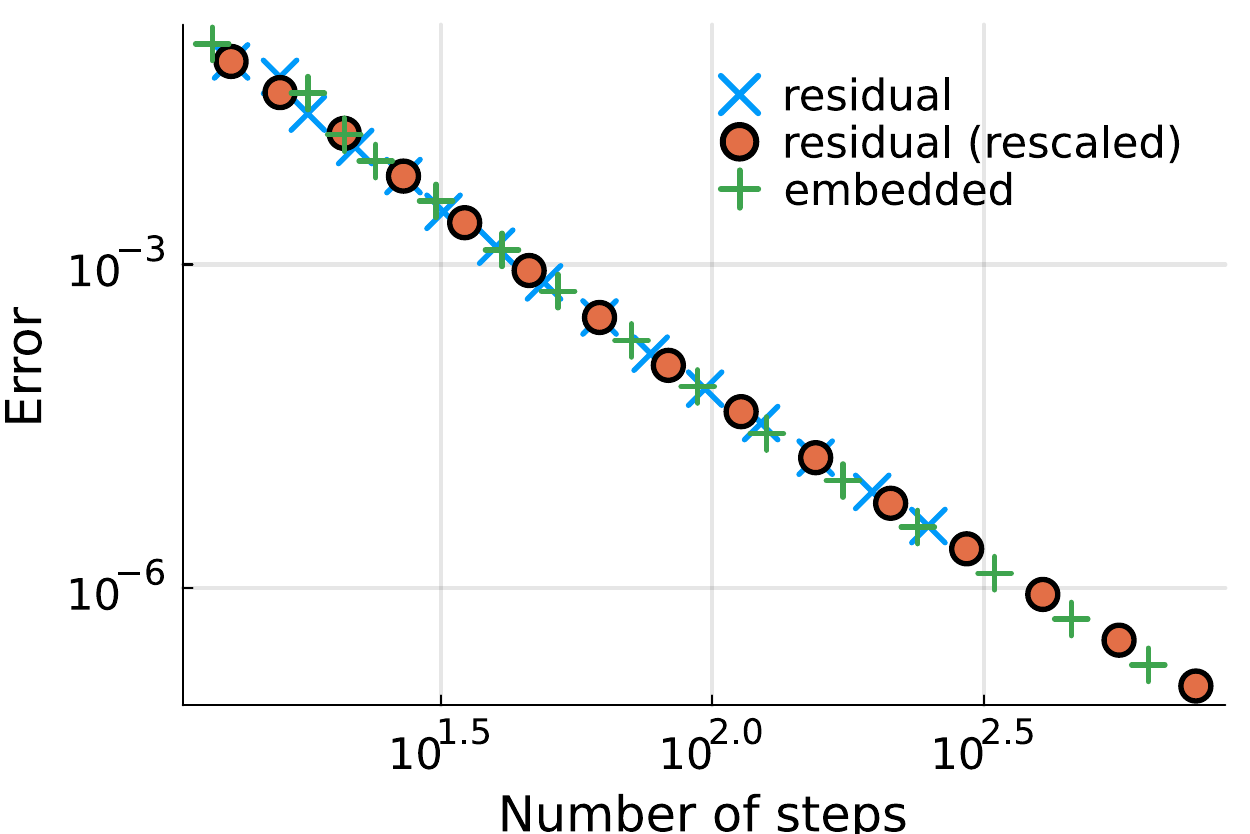}
    \caption{Work precision diagram.}
  \end{subfigure}%
  \caption{Results for the Euler equations of a rigid body using the
           third-order method of Bogacki and Shampine \cite{bogacki1989a32}
           with PI controller parameters $\beta = (0.6, -0.2)$.
           Since all results use an \EPS control, we also rescale the
           tolerance to achieve tolerance proportionality for the residual-based
           approach.}
  \label{fig:rigidbody}
\end{figure}

Numerical results obtained by the third-order method of Bogacki and Shampine
\cite{bogacki1989a32} with both the embedded method and the $L^1$ residual
error estimator are shown in Figure~\ref{fig:rigidbody}.
First, there are no issues with step rejections for this problem. Indeed,
the residual-based approach leads to no step rejections and the embedded method
rejects at most one or two steps for a few tolerances.

Next, we can observe the expected tolerance proportionality for the
embedded method with \EPS control. We also observe the expected scaling of
the global error as $\tau^{p / (p + 1)} = \tau^{3 / 4}$ for a given
tolerance $\tau$ and the residual-based
approach with \EPS control. Thus, we can achieve tolerance proportionality
by rescaling the tolerances as demonstrated in Figure~\ref{fig:rigidbody}.
Nevertheless, the residual-based approach with \EPS control leads to
tolerance convergence, even without rescaling.

Furthermore, the efficiency measured as the error for a fixed number of
(accepted plus rejected) steps is the same for all variants.
Indeed, all approaches lead to the same behavior in a classical
work precision diagram. Thus, both approaches lead to
\emph{tolerance convergence} (the error goes to zero for $\tol \to 0$) and
\emph{computational stability} (small changes in the tolerance lead to small
changes of the numerical results), see \cite{soderlind2002automatic} for a
discussion of these properties.

\subsection{1D Benjamin-Bona-Mahony equation}
\label{sec:bbm}

Next, we consider the Benjamin-Bona-Mahony (BBM) equation
\cite{benjamin1972model} (also known as regularized long wave equation)
\begin{equation}
\label{eq:bbm}
\begin{aligned}
  (\I - \partial_x^2) \partial_t u(t,x)
  + \partial_x \frac{u(t,x)^2}{2}
  + \partial_x u(t,x)
  &= 0,
  \\
  u(0, x) &= u^0(x),
\end{aligned}
\end{equation}
with periodic boundary conditions as a model of nonlinear dispersive wave
equations. We use the Fourier collocation semidiscretization of
\cite{ranocha2021broad} conserving discrete versions of the linear and
quadratic invariants
\begin{equation}
  \int u(t, x) \dif x,
  \quad
  \int \bigl( u(t,x)^2 + \bigl( \partial_x u(t,x) \bigr)^2 \bigr) \dif x.
\end{equation}
Due to the dispersive term $\partial_x^2 \partial_t u$ with mixed space and
time derivatives, the semidiscretization yields a non-stiff ODE with CFL
restriction of the form $\dt \lesssim \mathrm{const}$.

We consider the traveling wave solution
\begin{equation}
\label{eq:bbm-traveling-wave}
  u(t,x) = A \cosh( K (x - c t) ),
  \quad A = 3 (c - 1),
  \quad K = \frac{1}{2} \sqrt{1 - 1 / c},
\end{equation}
with speed $c = 1.2$ in the periodic domain $[-90, 90]$ and integrate the
semidiscretization with the third-order method of Bogacki and Shampine
\cite{bogacki1989a32} in a time interval big enough so that the wave traverses
the domain a bit more than once.

\begin{figure}[htb]
\centering
  \begin{subfigure}{0.49\textwidth}
  \centering
   \includegraphics[width=\textwidth]{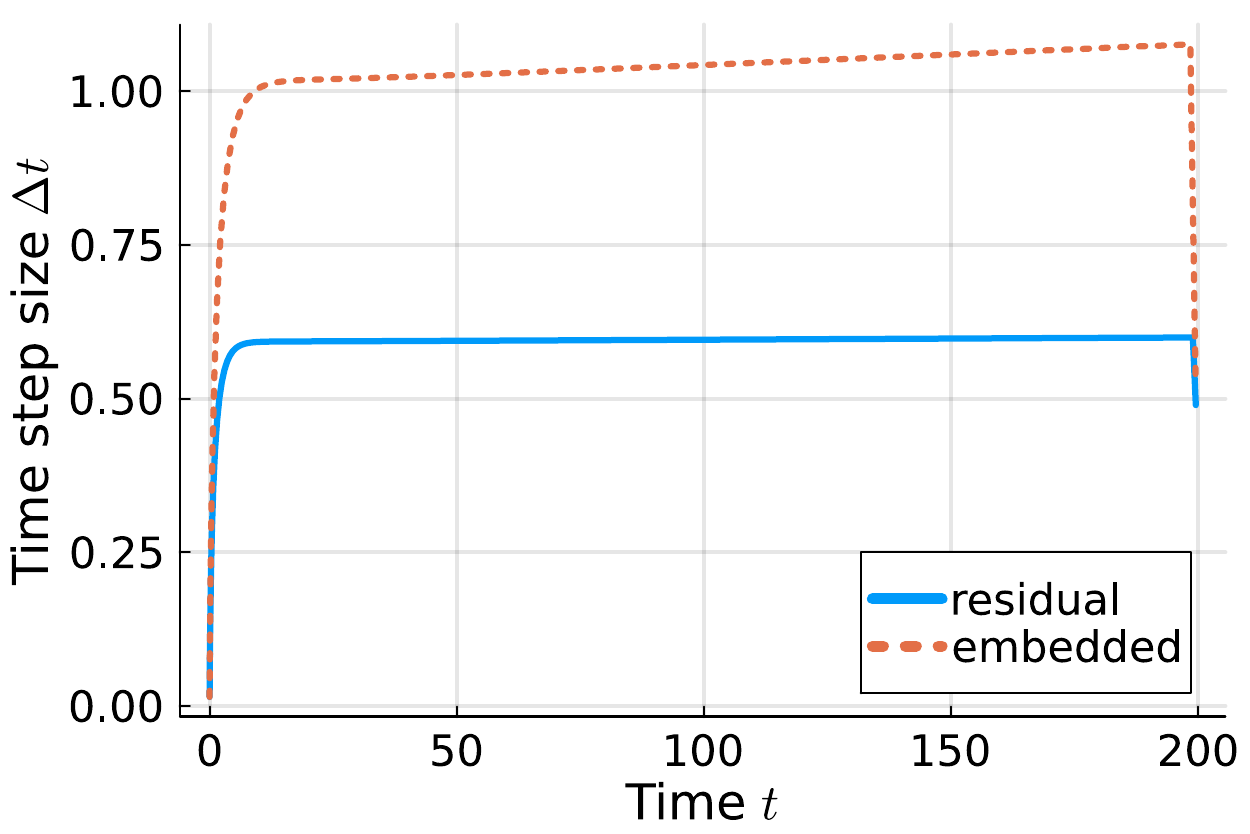}
    \caption{Time step size with tolerance $\tol = 10^{-4}$.}
    \label{fig:bbm_dt_vs_t}
  \end{subfigure}%
  \hspace*{\fill}
  \begin{subfigure}{0.49\textwidth}
  \centering
   \includegraphics[width=\textwidth]{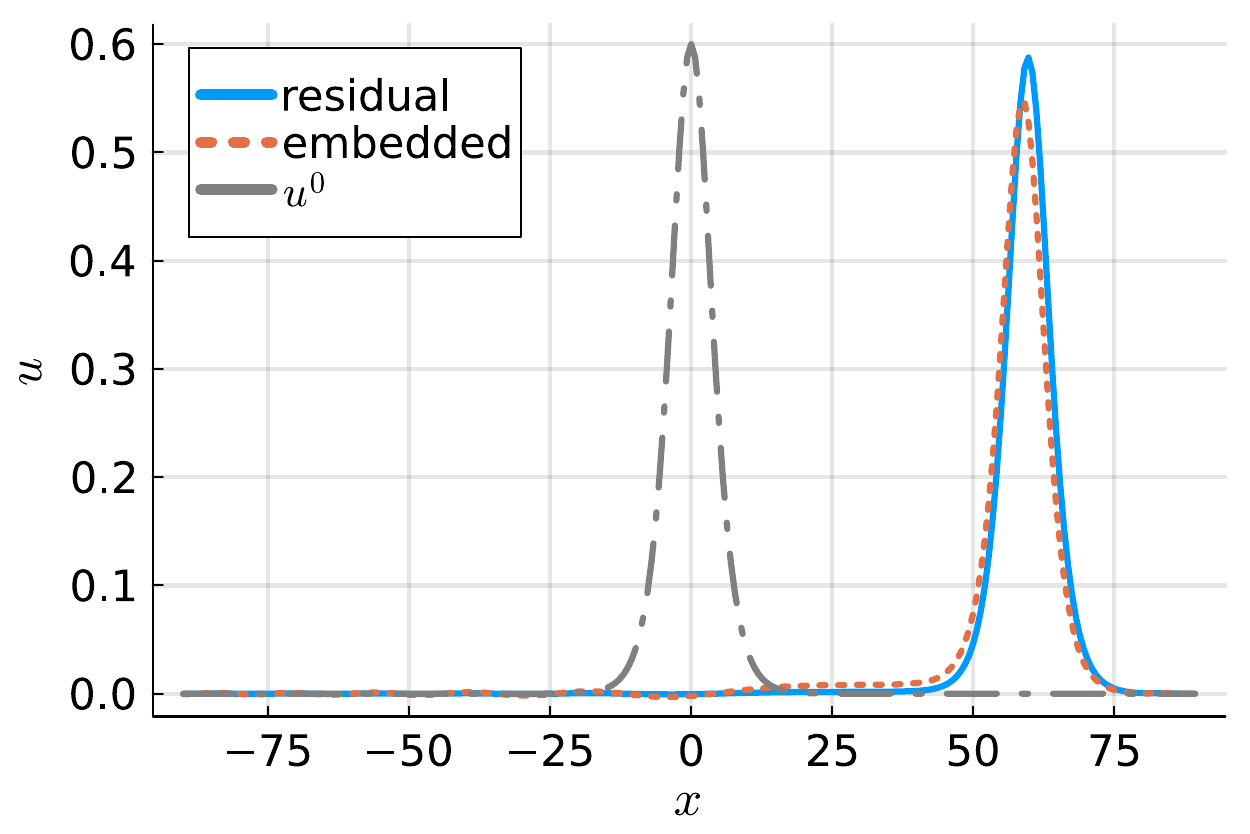}
    \caption{Numerical solutions with $\tol = 10^{-4}$.}
    \label{fig:bbm_solutions}
  \end{subfigure}%
  \\
  \begin{subfigure}{0.49\textwidth}
  \centering
   \includegraphics[width=\textwidth]{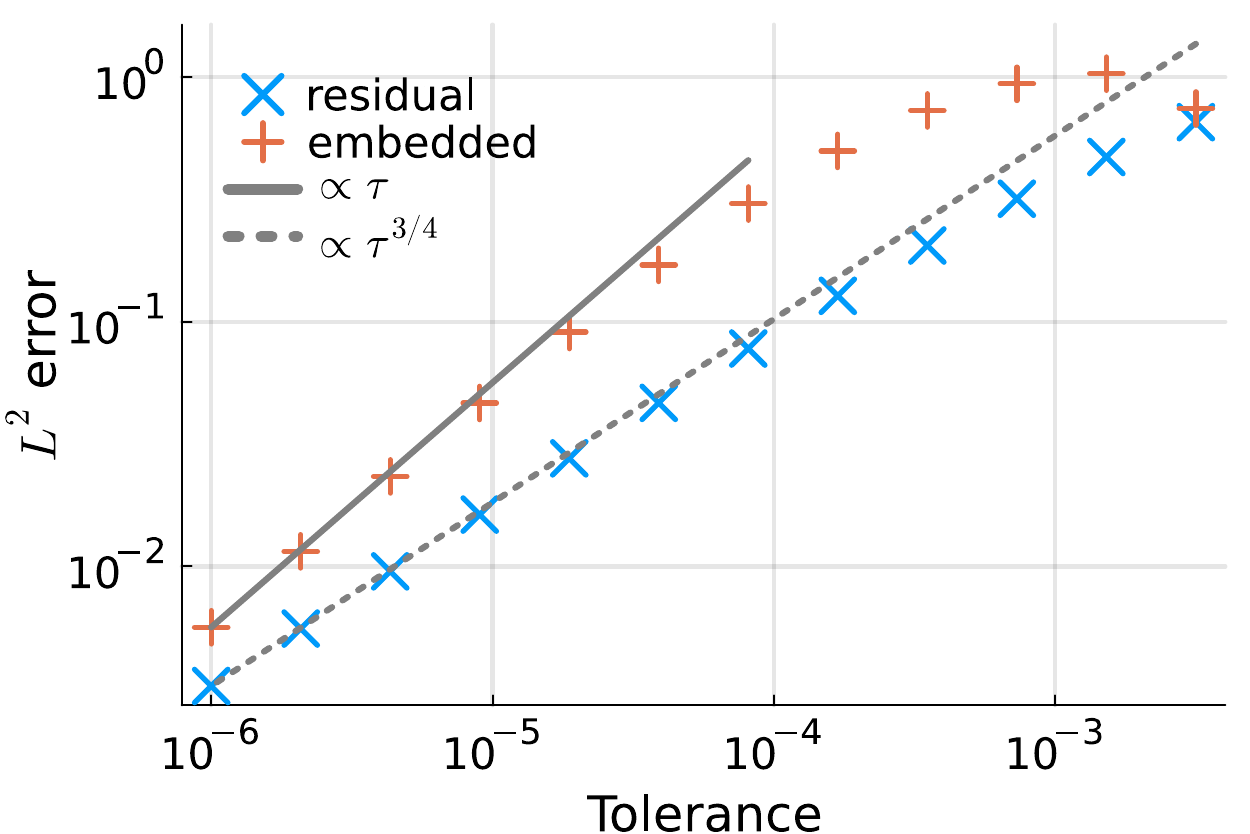}
    \caption{Discrete $L^2$ error for different tolerances.}
  \end{subfigure}%
  \hspace*{\fill}
  \begin{subfigure}{0.49\textwidth}
  \centering
   \includegraphics[width=\textwidth]{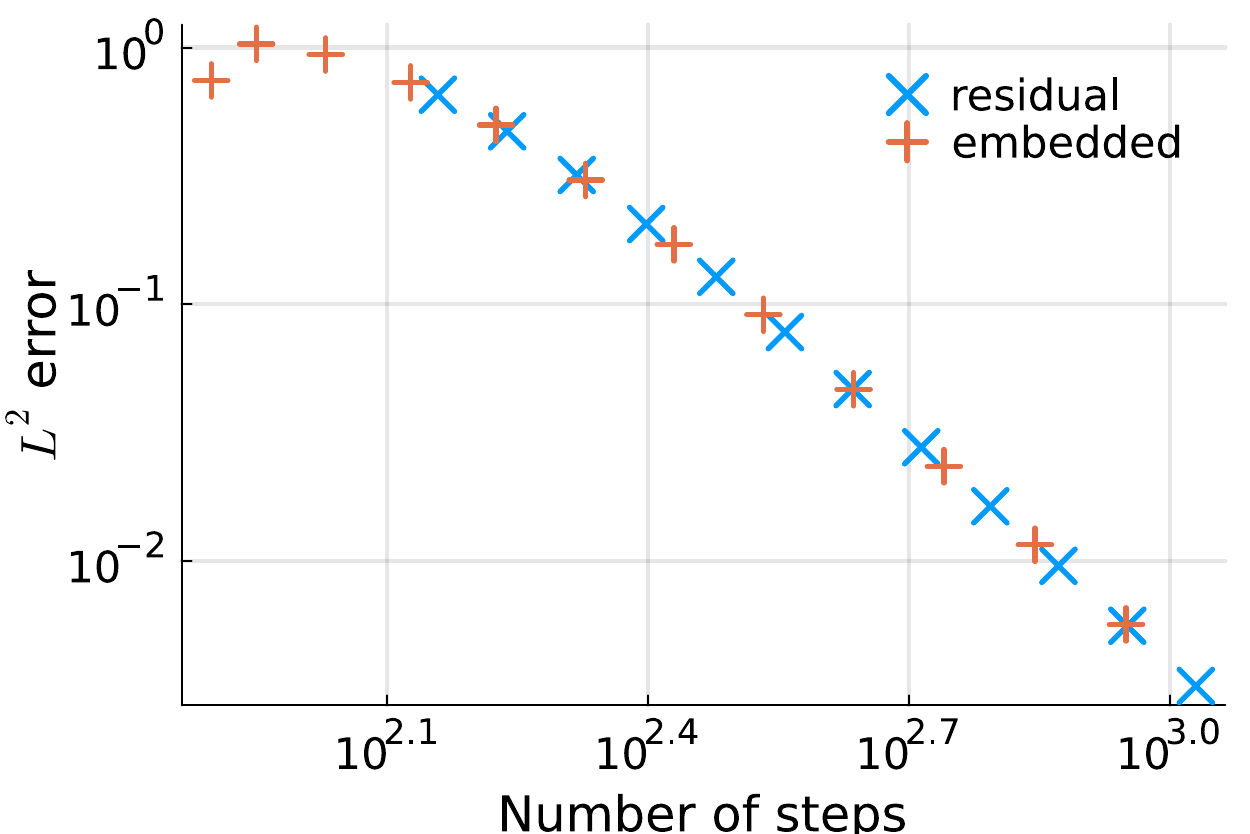}
    \caption{Work precision diagram.}
  \end{subfigure}%
  \caption{Results for the BBM equation with Fourier semidiscretization
           in space and the third-order method of
           Bogacki and Shampine \cite{bogacki1989a32} in time with
           PI controller parameters $\beta = (0.6, -0.2)$.}
  \label{fig:bbm}
\end{figure}

The results are visualized in Figure~\ref{fig:bbm}. The residual-based approach
leads to the expected behavior of a short transient period and a constant time
step size afterwards. The embedded method behaves similarly but leads to a
slowly growing time step size. The reason for this appears to be that it takes
slightly bigger time steps, leading to more dissipation of the numerical
solution. This in turn allows to take even bigger time steps, leading to
a visibly more dissipated numerical solution.

Both methods lead to an expected convergence behavior for stricter tolerances.
In particular, we get the tolerance proportionality for the embedded method
and a scaling of the global error as $\tau^{p / (p + 1)} = \tau^{3 / 4}$
for the residual-based approach with \EPS control.
Moreover, both methods lead to the same behavior in a
work precision diagram measuring the discrete $L^2$ error at the final time
for a given number of steps --- while both methods lead to no rejected steps
in this case.

\subsection{Reliability of the global error estimate}
\label{sec:one-sided-lipschitz}

Next, we study the reliability of the global a posteriori error estimate.
We consider the linear equation
\begin{equation}
\label{eq:one-sided-lipschitz-linear}
  u'(t) = u(t),
  \;
  t \in (0, 1),
  \qquad
  u(0) = 1,
\end{equation}
with (one-sided) Lipschitz constant $L = 1$.
Furthermore, we consider the nonlinear problem
\begin{equation}
\label{eq:one-sided-lipschitz-nonlinear}
  u'(t) = \exp\bigl( -u(t) \bigr),
  \;
  t \in (0, 100),
  \qquad
  u(0) = 1,
\end{equation}
with one-sided Lipschitz constant $L = 0$ (since $u \mapsto \e^{-u}$
is monotonically decreasing).

We measure the error of numerical solutions at the final time and
use the residual-based $L^2$ error estimate for \EPS control of the
time step size with the PID controller given by $\beta = (0.6, -0.2)$.
The results for Heun's second-order method and the third-order method
of Bogacki and Shampine \cite{bogacki1989a32} are shown in
Figure~\ref{fig:one-sided-lipschitz}.
First, we observe the expected scaling of the global error as
$\tau^{p / (p + 1)}$ in all cases. Moreover, we see that the global error
estimate of Lemma~\ref{lem:estimate} is reliable since it yields an upper
bound on the error in all cases.

\begin{figure}[htb]
\centering
  \begin{subfigure}{0.49\textwidth}
  \centering
   \includegraphics[width=\textwidth]{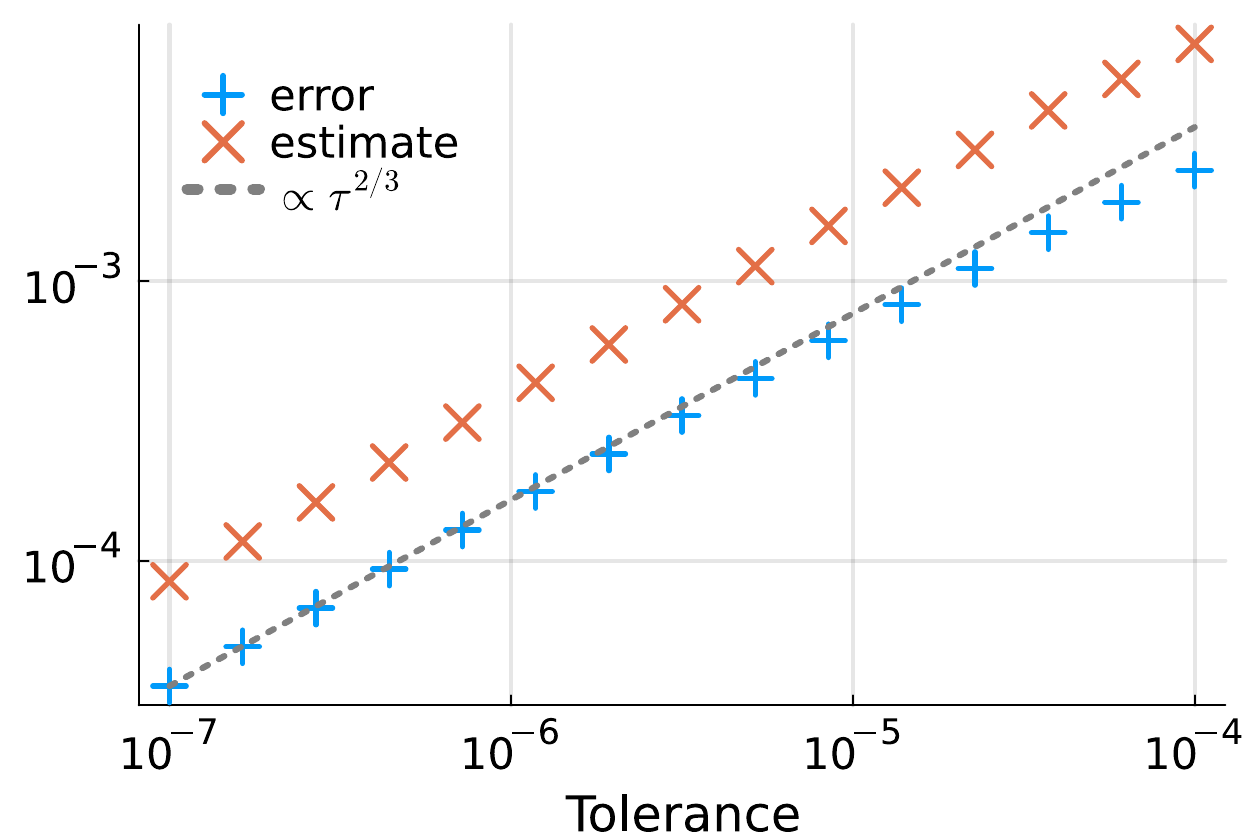}
    \caption{Linear problem \eqref{eq:one-sided-lipschitz-linear},
             second-order method of Heun.}
  \end{subfigure}%
  \hspace*{\fill}
  \begin{subfigure}{0.49\textwidth}
  \centering
   \includegraphics[width=\textwidth]{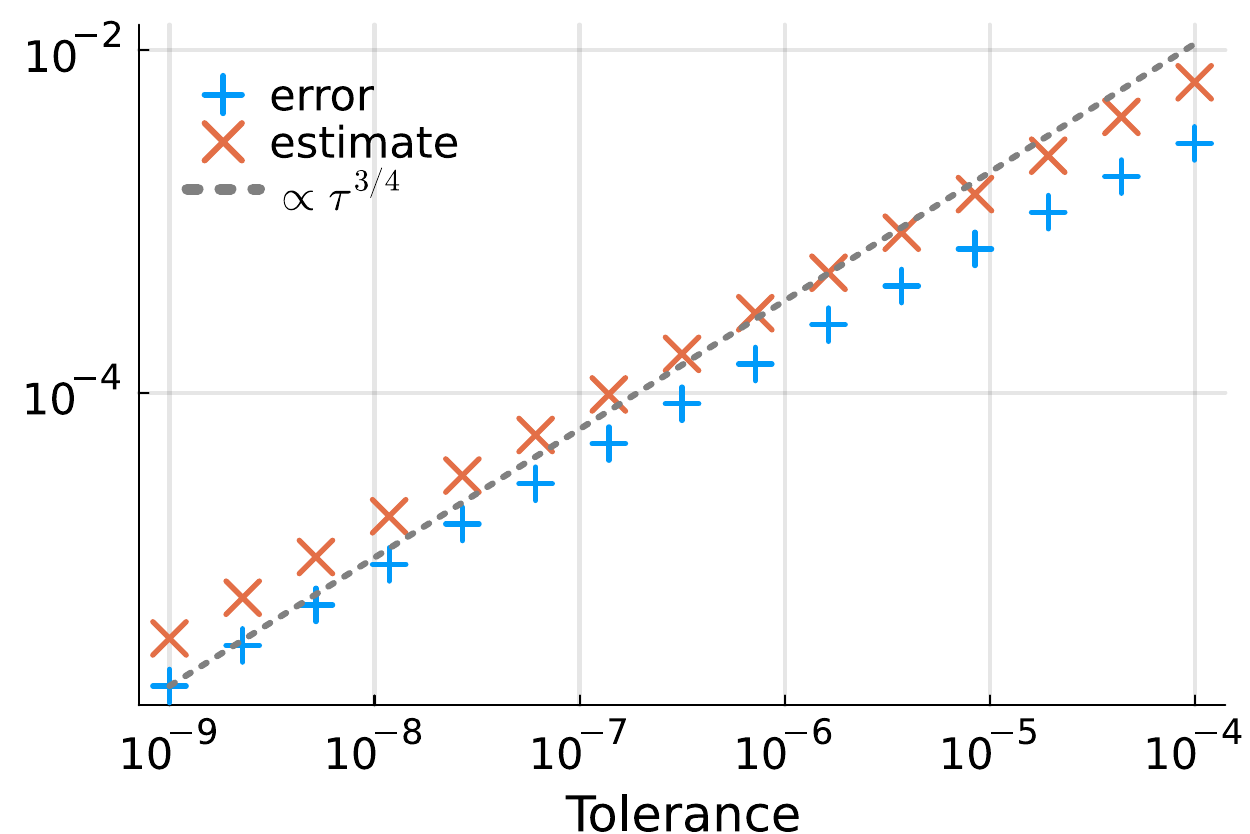}
    \caption{Linear problem \eqref{eq:one-sided-lipschitz-linear},
             third-order method of Bogacki and Shampine \cite{bogacki1989a32}.}
  \end{subfigure}%
  \\
  \begin{subfigure}{0.49\textwidth}
  \centering
   \includegraphics[width=\textwidth]{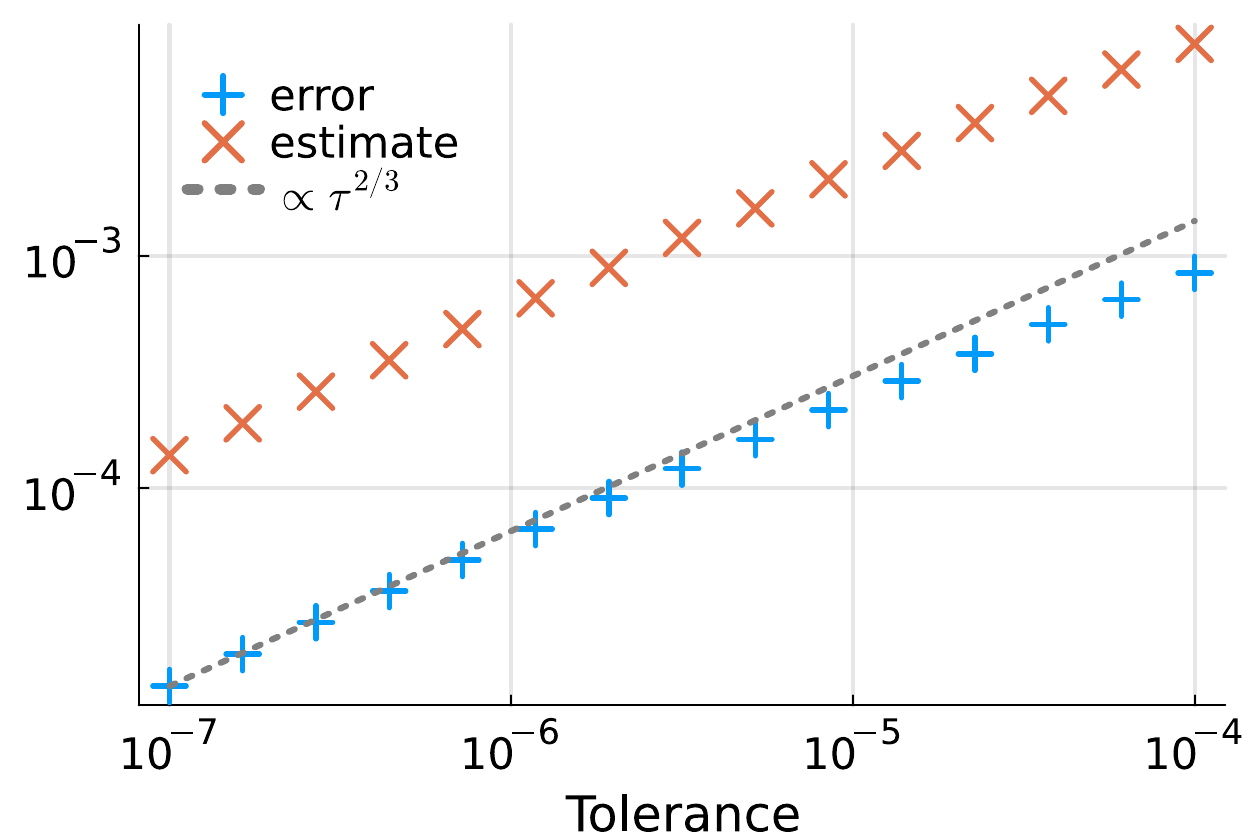}
    \caption{Nonlinear problem \eqref{eq:one-sided-lipschitz-nonlinear},
             second-order method of Heun.}
  \end{subfigure}%
  \hspace*{\fill}
  \begin{subfigure}{0.49\textwidth}
  \centering
   \includegraphics[width=\textwidth]{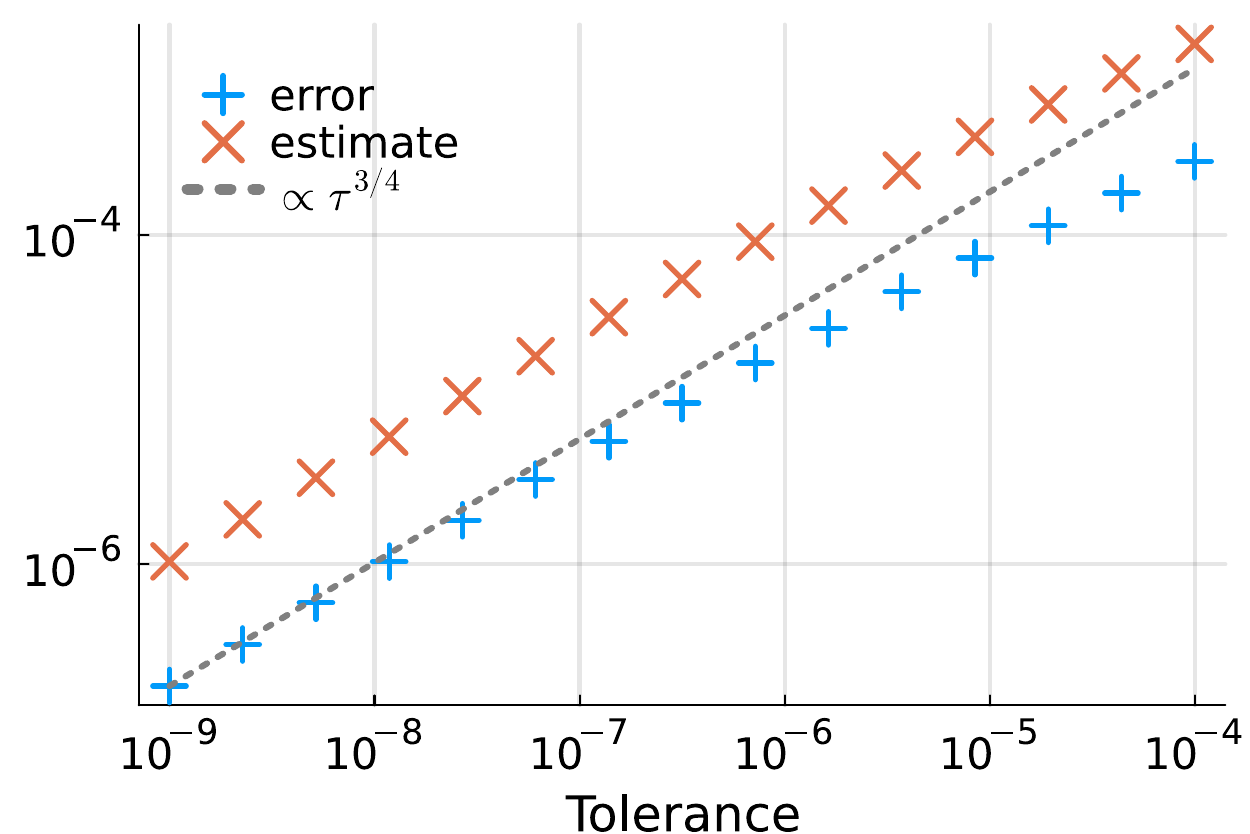}
    \caption{Nonlinear problem \eqref{eq:one-sided-lipschitz-nonlinear},
             third-order method of Bogacki and Shampine \cite{bogacki1989a32}.}
  \end{subfigure}%
  \caption{Comparison of the global error estimate at the final time
           obtained from Lemma~\ref{lem:estimate} with residual-based
           $L^2$ estimates and one-sided Lipschitz constants. The time step
           size is adapted using the same error estimates and the PI controller
           with parameters $\beta = (0.6, -0.2)$.}
  \label{fig:one-sided-lipschitz}
\end{figure}

\subsection{2D linear advection}
\label{sec:linadv}

We consider the 2D linear advection equation
\begin{equation}
  \partial_t u + \div(a u) = 0
\end{equation}
with periodic boundary conditions in $[-1, 1]^2$, the advection velocity
$a = (1, 1)^T$, and a sinusoidal initial condition. Using the method of lines
approach, we discretize the PDE first in space with a discontinuous Galerkin
spectral element method (DGSEM) on Gauss-Lobatto-Legendre nodes representing
polynomials of degree $p = 4$; an introduction to such nodal DG methods is
given in the textbooks \cite{hesthaven2007nodal,kopriva2009implementing}.
We divide the domain into $8^2$ uniform elements and apply the
local Lax-Friedrichs/Rusanov flux at interfaces.
Finally, we integrate the resulting ODE in time using the third-order method
of Bogacki and Shampine \cite{bogacki1989a32} with PI controller parameters
$\beta = (0.6, -0.2)$.

\begin{figure}[htbp]
\centering
  \begin{subfigure}{0.49\textwidth}
  \centering
   \includegraphics[width=\textwidth]{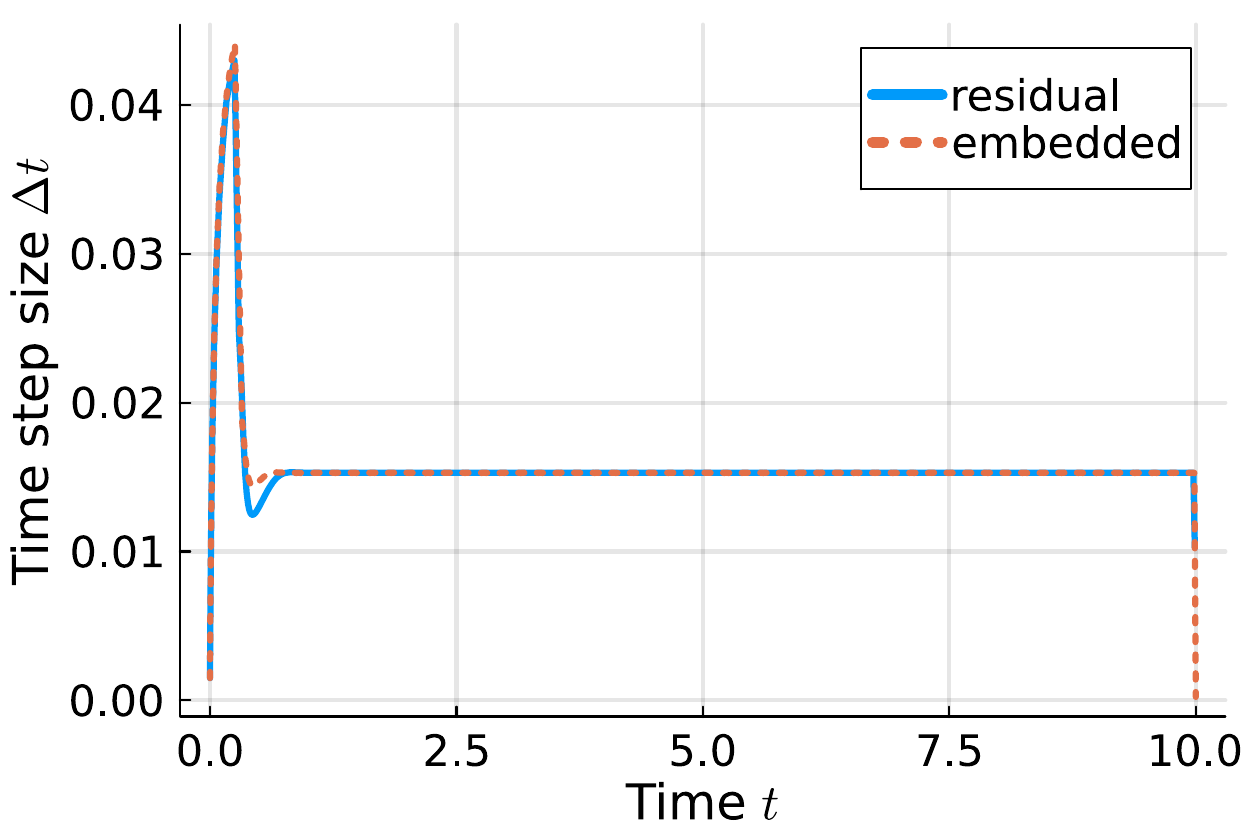}
    \caption{Time step size with tolerance $\tol = 10^{-4}$.}
    \label{fig:linadv_dt_vs_t}
  \end{subfigure}%
  \hspace*{\fill}
  \begin{subfigure}{0.49\textwidth}
  \centering
   \includegraphics[width=\textwidth]{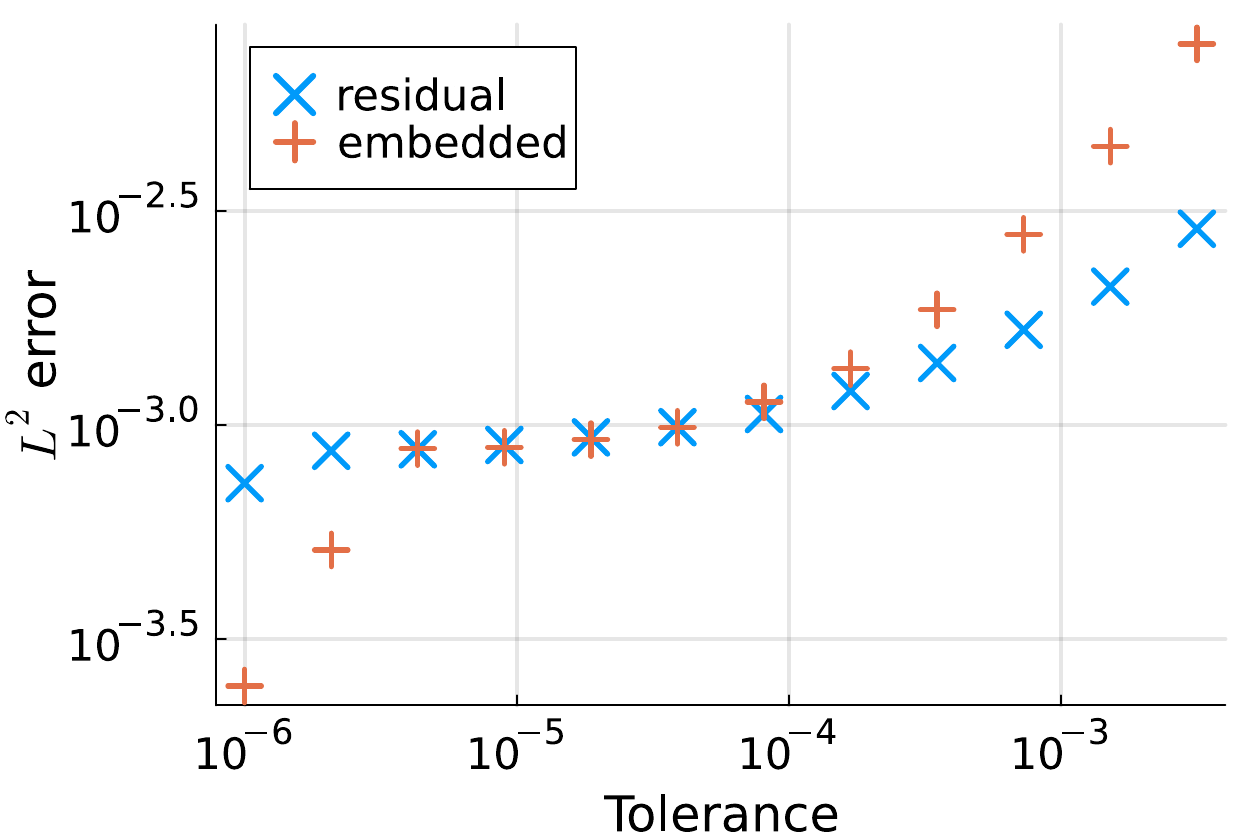}
    \caption{Discrete $L^2$ error for different tolerances.}
  \end{subfigure}%
  \\
  \begin{subfigure}{0.49\textwidth}
  \centering
   \includegraphics[width=\textwidth]{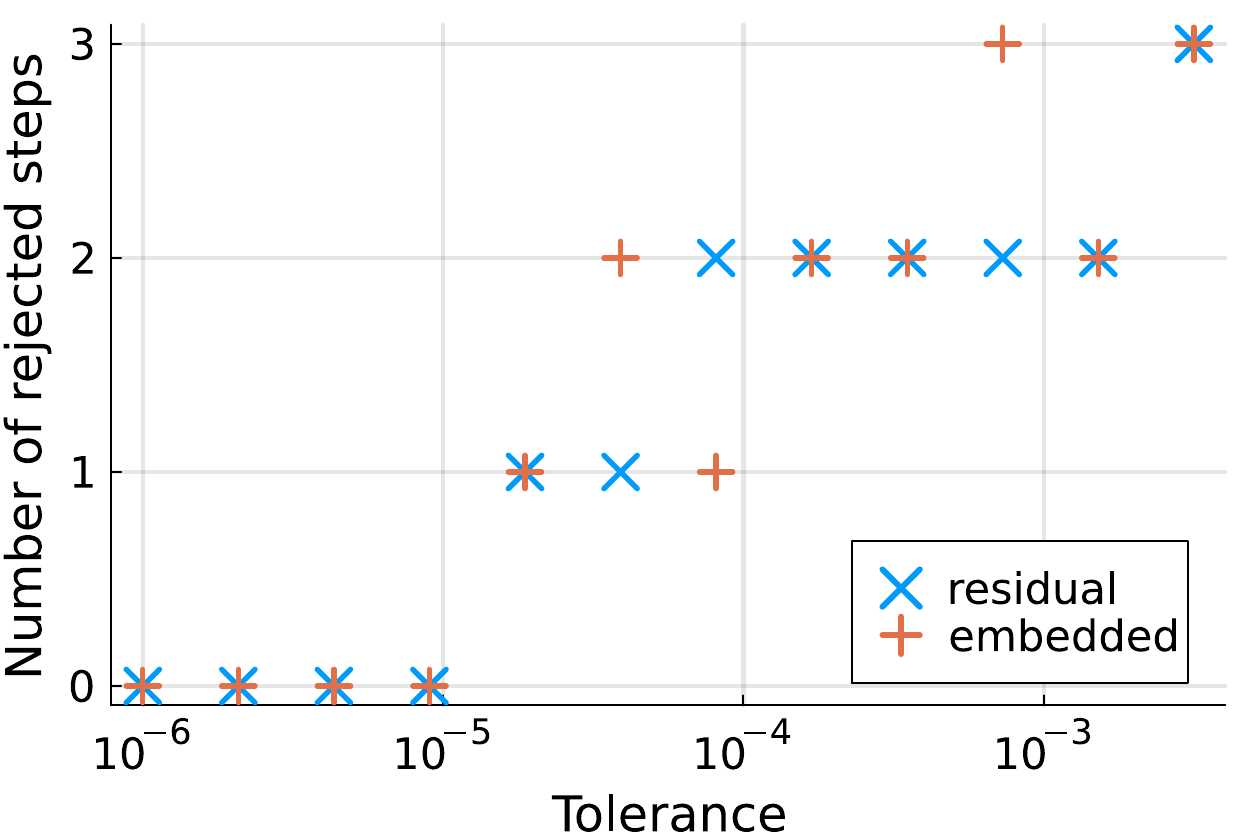}
    \caption{Number of rejected steps.}
  \end{subfigure}%
  \caption{Results for the linear advection problem with DGSEM
           semidiscretization in space and the third-order method of
           Bogacki and Shampine \cite{bogacki1989a32} in time with
           PI controller parameters $\beta = (0.6, -0.2)$.}
  \label{fig:linadv}
\end{figure}

As discussed in \cite{ranocha2021optimized,ranocha2023error}, a typical
behavior for such problems is as follows. Initially, the time step size
varies a bit and quickly converges to a constant. There is usually a range
of tolerances where the time step size $\dt$ is restricted by stability.
In this regime, error-based step size control with stable controllers results
in optimal time step sizes that can also be obtained by manually optimizing a
Courant-Friedrichs-Lewy (CFL) factor.
This behavior is shown in Figure~\ref{fig:linadv}; the embedded method and
the residual-based error estimator behave very similarly.
In particular, both methods lead to at most three rejected steps for loose
tolerances.

\subsection{3D inviscid Taylor-Green vortex}
\label{sec:euler}

Next, we consider the ideal 3D compressible Euler equations to simulate
an inviscid Taylor-Green vortex. We choose the initial condition
\begin{equation}
\begin{gathered}
  \rho = 1, \;
  v_1 =  \sin(x_1) \cos(x_2) \cos(x_3), \;
  v_2 = -\cos(x_1) \sin(x_2) \cos(x_3), \;
  v_3  = 0, \\
  p = \frac{\rho}{\mathrm{Ma}^2 \gamma} + \rho \frac{\cos(2 x_1) \cos(2 x_3) + 2 \cos(2 x_2) + 2 \cos(2 x_1) + \cos(2 x_2) \cos(2 x_3)}{16},
\end{gathered}
\end{equation}
where $\mathrm{Ma} = 0.1$ is the Mach number, $\rho$ the density, $v$ the
velocity, and $p$ the pressure.
We consider the domain $[-\pi, \pi]^3$ with periodic boundary conditions
and a time interval $[0, 10]$.
We use the entropy-stable semidiscretization with flux differencing DGSEM
with polynomials of degree $p = 3$, the entropy-conservative flux of
\cite{ranocha2018comparison,ranocha2020entropy,ranocha2021preventing}
in the volume and a local Lax-Friedrichs/Rusanov numerical flux at interfaces.
This kind of flux differencing discretizations is described in
\cite{fisher2013high,gassner2016split}; see also \cite{ranocha2023efficient}.

\begin{figure}[htbp]
\centering
  \begin{subfigure}{0.49\textwidth}
  \centering
   \includegraphics[width=\textwidth]{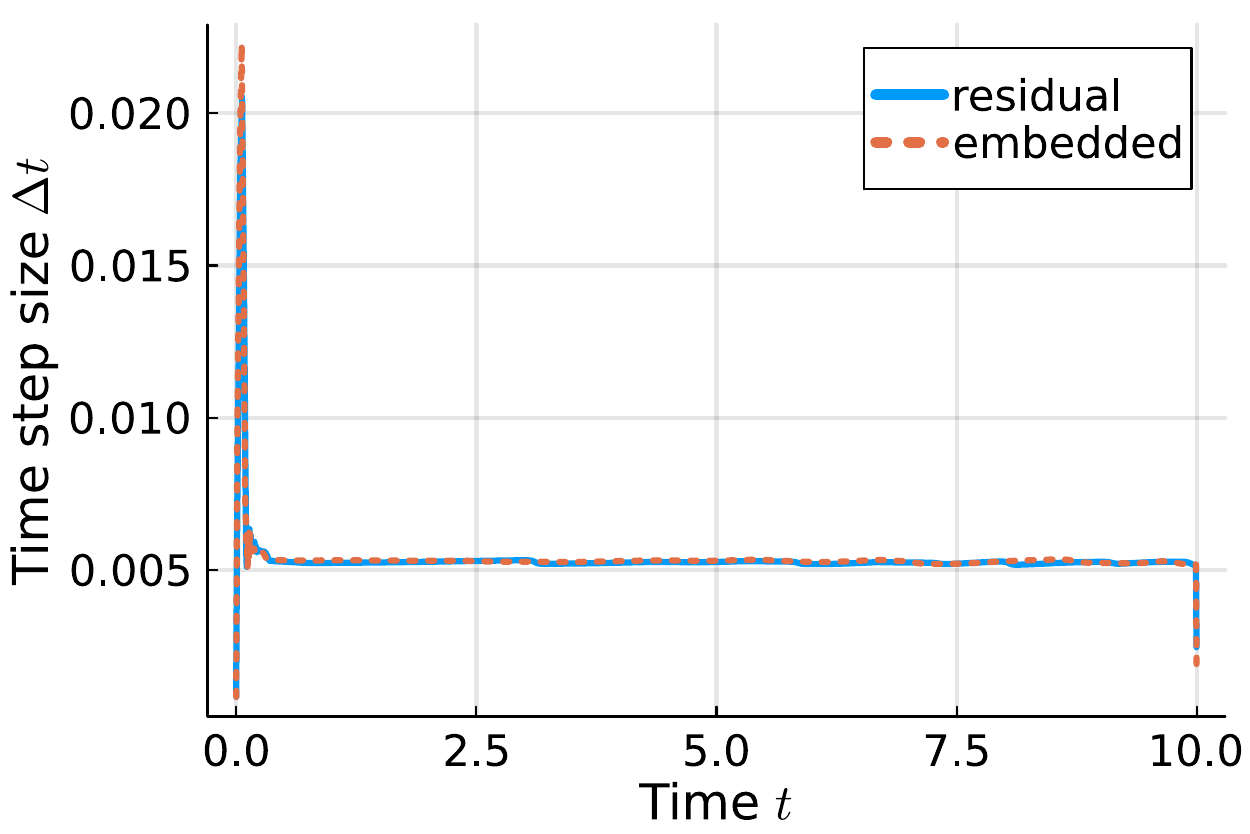}
  \end{subfigure}%
  \caption{Results for the 3D compressible Euler equations setup for an
           inviscid Taylor-Green vortex with entropy-stable flux differencing
           DGSEM in space and the third-order method of
           Bogacki and Shampine \cite{bogacki1989a32} in time with
           PI controller parameters $\beta = (0.6, -0.2)$ and tolerance
           $\tol = 10^{-5}$.}
  \label{fig:euler}
\end{figure}

The results of a simulation with $8^3$ elements and a time integration
tolerance $\tol = 10^{-5}$ are shown in Figure~\ref{fig:euler}. In this case,
the time step size $\dt$ is again restricted by stability constraints.
The time step sizes chosen by the different estimators --- the embedded method
and the $L^1$ residual estimate --- are visually indistinguishable.
The residual-based approach leads to 3 step rejections while the embedded
method leads to 2 rejected steps.

\section{Summary and conclusions}
\label{sec:summary}

We have analyzed stability of step size control of
explicit Runge-Kutta methods for ODEs using residual-based a posteriori error estimators of
\cite{dedner2016posteriori} when step sizes are dictated by stability and not accuracy.
It turned out that the situation is comparable to the case of embedded
methods, i.e., the classical I controller does not lead to step size control
stability while more advanced PI and PID controllers can be designed to be
stable. We have analyzed the situation for ODEs and demonstrated that the
findings extend to some PDEs discretized using the method of lines.
In particular, we have considered the nonlinear dispersive Benjamin-Bona-Mahony
equation as well as discontinuous Galerkin semidiscretizations of the 2D linear
advection and 3D compressible Euler equations.

The general behavior of the methods using estimates obtained from the
residual-based approach and embedded methods was comparable in the numerical
experiments. Thus, the main difference is that the residual-based approach
leads additionally to a rigorous a posteriori estimate of the global error
while the embedded methods are computationally cheaper.

\section*{Acknowledgments}

HR was supported by the Deutsche Forschungsgemeinschaft
(DFG, German Research Foundation, project number 513301895)
and the Daimler und Benz Stiftung (Daimler and Benz foundation,
project number 32-10/22).

\printbibliography

\end{document}